\documentclass[titlepage]{amsart}
\usepackage{amsmath}
\usepackage{amsthm}
\usepackage{amssymb}
\usepackage{setspace}
\usepackage{enumerate}
\usepackage[capitalise]{cleveref}
\usepackage{amsaddr}
\usepackage{xcolor}

\newtheorem{theorem}{Theorem}[section]
\newtheorem{lemma}{Lemma}[section]
\newtheorem{definition}{Definition}[section]

\newtheorem{corollary}{\bf Corollary}[section]
\newtheorem{Question}{\bf Question}[section]

\newcommand{\changed}[1]{\textcolor{black}{#1}}

\newcommand{\GH}{\mathrm{GH}}
\newcommand{\REG}{\mathrm{REG}}

\newcommand{\rk}{\mathrm{rk}}

\newcommand{\cf}{\mathrm{cf}}
\newcommand{\ran}{\mathrm{ran}}

\newcommand{\nacc}{\mathrm{nacc}}
\newcommand{\acc}{\mathrm{acc}}
\newcommand{\cof}{\mathrm{cof}}
\newcommand{\Tr}{\mathrm{Tr}}
\newcommand{\NS}{\mathrm{NS}}
\newcommand{\Mahlo}{\mathrm{Mahlo}}
\newcommand\blfootnote[1]{%
  \begingroup
  \renewcommand\thefootnote{}\footnote{#1}%
  \addtocounter{footnote}{-1}%
  \endgroup
}

\title{Reflecting on Inaccessible J\'onsson Cardinals}
\author{Shehzad Ahmed}
\address{Ohio University\blfootnote{\changed{The revision of this paper took place at Bar-Ilan University, where the author is currently a postdoctoral fellow and being supported by the Zuckerman Postdoctoral Scholars Program as well as the Israel Science Foundation (grant agreement 2066/18).}} \\
	Dept of Math 321 Morton Hall \\
	Athens, Ohio 45701-2979}
\email{sa066513@ohio.edu}

\begin{document}

\begin{abstract}
Given an inaccessible J\'onsson cardinal $\lambda$, a sequence of results due to Shelah from \cite{Sh} and \cite{Sh413} tell us that $\lambda$ must be at least $\lambda\times\omega$-Mahlo. We may then ask ourselves whether we can improve this bound, and unfortunately not much progress has been made on this question since \cite{Sh413}. Part of the problem may be that the proofs involved are rather complicated and make use of somewhat ad hoc machinery. In order to remedy this, we present a survey and new proofs of these results regarding the Mahlo degree of an inaccessible J\'onsson cardinal with the hope of making this material accessible.
\end{abstract}

\maketitle

\section{Introduction}

We begin by recalling the definition of a J\'onsson cardinal, which will be our primary object of study.

\begin{definition}
	\changed{Given cardinals $\lambda$ and $\theta$, we write $\lambda\to[\lambda]^{<\omega}_\theta$ to mean that  for every coloring $F:[\lambda]^{<\omega}\to\theta$, there exists some $A\in [\lambda]^\lambda$ such that $\ran(F\upharpoonright[A]^{<\omega})\subsetneq \theta$. We refer to the set $A$ as a {\em weakly homogeneous} set.}
\end{definition}

\begin{definition}
\changed{We say that an infinte cardinal $\lambda$ is {\em J\'onsson} if $\lambda\to[\lambda]^{<\omega}_\lambda$.}
\end{definition}

It is relatively well known that J\'onssonness is a large cardinal property situated between measurability and the existence of $0^\sharp$ (see e.g. \cite{Ka}). It is, however, a rather mysterious large cardinal property in that not much is known about the behavior of J\'onsson cardinals. For example, the following questions are open:\\

\begin{enumerate}
	\item Is it consistent that there is a singular cardinal $\mu$ such that $\mu^+$ is J\'onsson?
	\item Is it consistent that $\aleph_\omega$ is J\'onsson?
	\item Is there an inaccessible J\'onsson cardinal $\lambda$ which is \changed{greatly Mahlo}?
	\item If $\lambda$ is inaccessible and J\'onsson, how often must stationary subsets of $\lambda$ reflect?\\
\end{enumerate}

Our concern will be primarily with the last question. Our first result along these lines comes from Tryba and Woodin independently (see \cite{Eis1} for a proof).

\begin{theorem}[Tryba-Woodin]\label{Tryba-Woodin}
	If $\lambda$ is a regular J\'onsson cardinal, then every stationary subset of $\lambda$ must reflect at points of arbitrarily large cofinality below $\lambda$.
\end{theorem}

From here, we have two questions we would like to consider in the case that $\lambda$ is a regular J\'onsson cardinal:\\

\begin{enumerate}
	\item How much simultaneous reflection does $\lambda$ enjoy?
	\item How often do stationary subsets of $\lambda$ reflect?\\
\end{enumerate}

In the case that $\mu$ is singular and $\mu^+$ is J\'onsson, \changed{then any fewer than $\cf(\mu)$-many stationary subsets of $\mu^+$ reflect simultaneously (this follows from work of Eisworth in \cite{Eis3})}. Here, we will consider these questions only in the context of inaccessible J\'onsson cardinals, where the known results seem very sparse.

Shelah has shown, in \cite{Sh413}, that if $\lambda$ is an inaccessible J\'onsson cardinal, then $\lambda$ must be $\lambda\times\omega$-Mahlo. Beyond that, however, we are in the dark. One potential culprit for the lack of progress on this problem is that the proof of the aforementioned result is rather complicated. Our goal here is to survey and provide new proofs of several known results, all of which are due to Shelah from \cite{Sh} or \cite{Sh413}, regarding the Mahlo degree of an inaccessible J\'onsson cardinal. 

In order to do this, we begin by collecting a number of known results regarding J\'onsson cardinals and club guessing in section 2. In particular, we isolate the notion of a J\'onsson model and recast several common techniques as lemmas about J\'onsson models. While this may seem like an unimportant distinction, this recasting will allow us to further analyze the relationship between club guessing sequences and J\'onsson models.

In section 3, we briefly survey Shelah's notion of the rank of a stationary set from Chapter IV of \cite{Sh} and \cite{Sh506}. It turns out that this notion essentially comes \changed{from} the iterated trace operator, which was introduced by Jech in \cite{Jech84}. The main result of this section is that, the rank of an increasing union of sets $\langle S_i : i<\theta\rangle$ can be bounded by the ranks of the sets $S_i$ plus some error coming from the Galvin-H\changed{a}jnal rank of a particular function.

In section 4, we develop the machinery of \changed{co}-J\'onsson filters, and explore their connection with stationary reflection. Of note is that \changed{co}-J\'onsson filters can be shown to exist on J\'onsson successors of singulars, and that many of the combinatorial consequences of such a cardinal can be derived directly from the existence of a J\'onsson filter. In other words, any \changed{J\'onsson} cardinal which has a \changed{co}-J\'onsson filter must behave similarly to a J\'onsson successor of a singular.

Finally, in section 5, we use the previously developed machinery to give a new, self-contained proof of Shelah's result from \cite{Sh413} that any inaccessible J\'onsson cardinal must be at least $\lambda\times\omega$-Mahlo. In \cite{Sh} and \cite{Sh413}, Shelah makes use of an ideal called $id^j(\bar C)$ (which is defined below), to put a lower bound on the Mahlo rank of an inaccessible J\'onsson cardinal.

\begin{definition}
	Suppose that $\lambda$ is an inaccessible J\'onsson cardinal, $S\subseteq \lambda$ is unbounded, and $\bar C=\langle C_\delta : \delta\in S\rangle$ is a sequence of sets with $C_\delta\subseteq\delta$. We define $id^j(\bar C)$ to be the family of sets $A$ which satisfy the following condition: For every regular $\chi>\lambda$, and $x\in H(\chi)$, there is an $n<\omega$ and a sequence $\bar M=\langle M_i : i<n\rangle$ such that:
	
	\begin{enumerate}
		\item For each $i<n$, $M_i$ is an elementary submodel of $H(\chi)$ such that:
		\begin{enumerate}
			\item $\lambda, x\in M_i$;
			\item $|M_i\cap\lambda|=\lambda$;
			\item $\lambda\not\subseteq M_i$;
		\end{enumerate} 
		\item For some $\alpha^*<\lambda$, for no $\delta\in S\setminus \alpha^*$ do we have:
		\begin{enumerate}
			\item $\delta=\sup(M_i\cap\delta)$ for $i<n$ and
			\item for every $\beta<\delta$, for some $\alpha$ we have: $\alpha\in nacc(C_\delta)\setminus \beta$, $\cf(\alpha)\geq\beta$, and for every $i<n$ we have that either $\changed{\alpha}\in M_i$ or $\min (M_i\setminus\alpha)$ is singular.
		\end{enumerate} 
	\end{enumerate} 
\end{definition}

We eschew the (somewhat ad hoc) use of this ideal in favor of leveraging \changed{co}-J\'onsson filters and the relationship between club guessing and J\'onsson models. The advantage of avoiding $id_j(\bar C)$ is that the proof becomes much clearer, and the interplay between J\'onssonness and Mahlo rank is brought to the forefront.

\subsection{Notation and Conventions}

\changed{Our notation and conventions are mostly standard, though we list the exceptions to this. For a regular cardinal $\kappa$, we use $\cof(\kappa)$ to denote the class of ordinals with cofinality greater than $\kappa$. We will use $\REG$ to denote the class of regular cardinals. Given an ordinal $\delta$ and a regular cardinal $\kappa$, we use $S^\delta_\kappa$ to denote the collection of ordinals below $\delta$ of cofinality $\kappa$. Adding inequalities (i.e. $\cof(>\kappa)$ have the obvious meaning. For us, an inaccessible cardinal is a {\em weakly} inaccessible cardinal, and similarly Mahlo cardinals {\em weakly} Mahlo. For any regular cardinal $\chi$, we will use $<_\chi$ to denote a well order of $H(\chi)$. For a set of ordinals $C$, we use $\acc(C)$ to denote the set of accumulation points of $C$ which are in $C$, and $\nacc(C)$ denotes the set of non-accumulation points. For an ordinal $\delta$ of uncountable cofinality, we denote the non-stationary ideal on $\delta$ by $\mathrm{NS}_\delta$. Given an ideal $I$, we use $I^+$ to denote the collection of $I$-positive sets. Finally, it's worth noting that ideals need not be proper for us.}

\section{J\'onsson Cardinals}

In this section, we collect a number of basic results regarding J\'onsson cardinals which we will be making use in later sections. We first recall an alternative characterization in terms of elementary submodels. The following appears as Theorem 5.3 of \cite{Eis1}.

\begin{lemma}[Folklore]\label{Jonsson model}
	For a cardinal $\lambda$, the following are equivalent:
	
	\begin{enumerate}
		\item $\lambda$ is J\'onsson.
		\item There is some regular $\chi>\lambda$ and $M\prec (H(\chi), \in, <_\chi)$ such that
		
		\begin{enumerate}
			\item $|M\cap\lambda|=\lambda$.
			\item $\lambda\in M$.
			\item $\lambda\not\subseteq M$.
		\end{enumerate}
	\end{enumerate}
\end{lemma}

In the proof of the above lemma, one can easily show that we have some control over the $M$ that we produce. In particular, we can also ask, for some $x\in H(\chi)$ that $x$ is also in $M$. This leads us to the following definition.

\begin{definition}
	Suppose that $\lambda$ is a J\'onsson cardinal, $\chi>\lambda$ is sufficiently large and regular, and $x\in H(\chi)$. We say that $M$ is an \changed{\em $x$-J\'onsson model} if the following hold:
	
	\begin{enumerate}
		\item $M\prec ((H(\chi),\in, <_\chi)$
		\item $|M\cap\lambda|=\lambda$.
		\item $\lambda, x\in M$.
		\item $\lambda\not\subseteq M$.
	\end{enumerate}
\end{definition}

If we do not need that a particular $x$ sits inside $M$, we will just say that $M$ is a J\'onsson model. All of our J\'onsson models will be elementary submodels of some sufficiently large and regular $\chi$. At this point, we would like to highlight the interaction between J\'onsson models and club guessing.

\begin{definition}
	Given an ordinal $\lambda$ of uncountable cofinality, and a stationary $S\subseteq\lambda$, we say that a sequence $\bar C=\langle C_\delta : \delta\in S\rangle$ is an \changed{\em $S$-club system} if each $C_\delta$ is a club subset of $\delta$.
\end{definition}

\changed{It is worth pointing out that $S$-club systems are also referred to as ladder systems}. We will be using these $S$-club systems to define a particular sort of club guessing found in \cite{Sh}. Note that we allow $\delta$ to be of countable cofinality, in which case $C_\delta$ simply needs to be unbounded in $\delta$. We now associate a sequence of ideals to a given $S$-club system.

\begin{definition}
	Suppose that $\lambda$ is a regular cardinal, $S\subseteq\lambda$ is stationary, and $\bar C=\langle C_\delta : \delta\in S\rangle$ is an $S$-club system. For each $\delta\in S$, we let $I_\delta$ be the ideal on $C_\delta$ generated by the following sets:\\
	
	\begin{enumerate}
		\item The \changed{accumulation} points of $C_\delta$,
		\item bounded subsets of $C_\delta$,
		\item and sets of the form $\{\alpha \in C_\delta : \cf(\alpha)<\gamma \}$ for each regular $\gamma<\delta$.\\
	\end{enumerate}
	
	We let $\bar I=\langle I_\delta : \delta\in S\rangle$.
\end{definition} 

In \cite{Sh}, Shelah did not fix a sequence of ideals ahead of time, and instead left open the option of using any sequence of ideals on $\bar C$. For our purposes, we will only need the above sequence. We can now define the sort of club guessing that we are interested in.

\begin{definition}
	Suppose that $\lambda$ is a regular cardinal, $S\subseteq\lambda$ is stationary, and $\bar C=\langle C_\delta : \delta\in S\rangle$ is an $S$-club system. We say that $\bar C$ is a \changed{\em club guessing sequence} if, for every club $E\subseteq \lambda$, there is a stationary $S'\subseteq S\cap E$ such that
	
	\begin{equation*}
	C_\delta\cap E\notin I_\delta
	\end{equation*}
	\\
	for each $\delta\in S'$. \changed{That is, for every regular $\gamma<\delta$ and every $\alpha<\delta$, there is some $\beta \in \nacc(C_\delta)\cap E$ such that $\beta > \alpha$ and $\cf(\beta)\geq \gamma$.}
\end{definition}

The above notion of club guessing turns out to be exactly what we need when working with J\'onsson cardinals. \changed{In particular, we will need to consider ideals coming from club guessing.}

\begin{definition}
	\changed{Suppose that $\lambda$ is a regular cardinal, $J$ is an ideal on $\lambda$ extending the non-stationary ideal, $S\subseteq\lambda$ is $I$-positive, and $\bar C=\langle C_\delta : \delta\in S\rangle$ is an $S$-club system. We define an ideal $I(\bar C, J)$ by saying that $A\in I(\bar C, J)$ if and only if the sequence $\bar C\cap A=\langle C_\delta\cap A : \delta\in S \rangle$ is not a club guessing sequence modulo $J$. That is, there is some club $E\subseteq \lambda$ such that the set}
	
	\begin{equation*}
	\changed{\{\delta\in S : \nacc(C_\delta)\cap A\cap E \notin I_\delta\}}
	\end{equation*}
	\\
	\changed{is in $J$}.
\end{definition}

\begin{definition}
	Suppose that $\lambda$ is a J\'onsson cardinal, and let $M$ be a J\'onsson model witnessing this. We define the map $\beta_M:\lambda\to M\cap \lambda$ by 
	
	\begin{equation*}
	\beta_M(\alpha)=\min(M\cap\lambda)\setminus\alpha.
	\end{equation*}
\end{definition}

Regarding the map $\beta_M$, note that if $\alpha\in M$, then $\beta_M(\alpha)=\alpha$. Otherwise, $\beta_M(\alpha)$ must be limit of uncountable cofinality. This leads us to the following useful fact about the map $\beta_M(\delta)$.

\begin{lemma}\label{Reflection at beta delta}
	Suppose $\lambda$ is J\'onsson, and let $M$ be a J\'onsson model for $\lambda$. If $\delta=\sup(M\cap\delta)$, $\delta\neq\beta_M(\delta)$, and $d\subseteq\beta_M(\delta)$ is club in $\beta_M(\delta)$ with $d\in M$, then $\delta\in\acc(d)$.
\end{lemma}

We will also often make use of the following result, which can be thought of as a companion to the previous.

\begin{lemma}\label{club indiscenibility}
	Let $\lambda$ be J\'onsson, and let $M$ be a J\'onsson model for $\lambda$. For $\alpha\in \lambda$, if $E\in M$ is a $\leq\beta_M(\alpha)$-closed set of ordinals with $\alpha\in E$, then $\beta_M(\alpha)\in E$.
\end{lemma}

\begin{proof}
	Suppose otherwise, then clearly $\alpha<\beta_M(\alpha)$ and $\sup(E\cap\beta_M(\alpha))\in M\cap\beta_M(\alpha)$ since $E$ is closed under sequences of length $\leq\beta_M(\alpha)$. But then, $\sup(E\cap\beta_M(\alpha))<\alpha$ by the minimality of $\beta_M(\alpha)$, which is contradiction.
\end{proof}

The previous two lemmas tell us that, if $M$ is a J\'onsson model, and $\delta\notin M$, then clubs in $M$ have difficulty distinguishing between $\delta$ and $\beta_M(\delta)$. Our next lemma provides us with the connection between J\'onsson models and club guessing. 

\begin{lemma}[Lemma 5.20 of \cite{Eis1}]\label{Swallowing Ladders}
	Suppose that $\lambda$ is a regular J\'onsson cardinal and $M$ is a witnessing J\'onsson model, and let
	
	\begin{equation*}
	E=\{\delta<\lambda : \delta=\sup (M\cap \delta)\}.
	\end{equation*}
	\\
	If $S\subseteq\lambda$ is stationary and $\bar C=\langle C_\delta : \delta\in S\rangle$ is an $S$-club system with $S,\bar C\in M$, then for every $\delta\in S\setminus M$, 
	
	\begin{equation*}
	\{\alpha\in \nacc(C_\delta) : \cf(\alpha)>\cf(\beta_M(\delta))\}\cap E\subseteq M.
	\end{equation*} 
	\\
	Further, if $\delta\in S\cap M$, then $\nacc(C_\delta)\cap E\subseteq M$.
\end{lemma}

Roughly speaking the above ladder swallowing trick allows us to sneak large portions of $\bar C$ inside J\'onsson models, provided we know that the ordinals $\beta_M(\delta)$ are singular often enough. At this point, we would like to note that club guessing sequences exist. This is essentially the content of Claim 2.6, Remark 2.6A, and Claim 2.7 from Chapter III of \cite{Sh}, which we quote a consequence of.

\begin{lemma}[Shelah]\label{club guessing}
	Suppose that $\lambda$ is either a successor of singular \changed{cardinal $\mu$} or \changed{$\lambda$ is} an inaccessible cardinal, and let $\kappa$ be an uncountable regular cardinal below $\lambda$. Suppose that $S$ is a stationary set of singular cardinals satisfying either
	
	\begin{enumerate}
		\item $S\subseteq\{\delta<\lambda : \cf(\delta)=\kappa\}$ if $\lambda=\mu^+$, or
		\item $S\subseteq \{\delta <\lambda : \omega<\cf(\delta)\leq\kappa\}$ if $\lambda$ is inaccessible.
	\end{enumerate}
	Then there is some $S'\subseteq S$ stationary and an $S'$-club system $\bar C=\langle C_\delta : \delta\in S'\rangle$ such that $\bar C$ is a club guessing sequence.
\end{lemma}

\section{Stationary Sets and Their Rank}

In order to further investigate the relationship between J\'onsson cardinals and stationary reflection, we want to define a measure of how often a stationary set reflects. We first consider an iterated version of the trace operator.

\begin{definition}
	Suppose that $S\subseteq\lambda$ for some ordinal $\lambda$ of uncountable cofinality. We define the \changed{\em trace operator} on $S$ by:
	
	\begin{equation*}
	\Tr(S)=\{\alpha<\lambda : S\cap\alpha\text{ is stationary in }\alpha\}.
	\end{equation*}
\end{definition}

The following version of the iterated trace operator is originally defined in \cite{Jech84}, and is essentially the same as the operator $A^{[\bar e,i]}$ defined in \cite{Sh506}.

\begin{definition}
	\changed{We then define the {\em iterated trace operator} on $S\subseteq\lambda$ for $\alpha<\lambda$ recursively as follows:}\\
	
	\begin{enumerate}
		\item \changed{$\Tr_0(S)=S$};
		\item \changed{$\Tr_{\alpha+1}(S)=\Tr(\Tr_\alpha(S))$};
		\item \changed{For $\alpha$ limit, we set:}
			\begin{equation*}
			\changed{\Tr_\alpha=\bigcap_{\beta<\alpha}\Tr_\beta(S)}\\
			\end{equation*}
	\end{enumerate}

	\changed{In the case that $\lambda$ is an inaccessible cardinal, we additionally fix a $\lim(\lambda^+)\setminus\lambda$-club system $\bar e=\langle e_\alpha : \alpha\in\lim(\lambda^+), \lambda\leq\alpha<\lambda^+\rangle$, and define the trace for ordinals $\lambda\leq \alpha<\lambda^+$ as follows:}\\

	\begin{enumerate}
		\item \changed{$\Tr_{\alpha+1}(S)=\Tr(\Tr_\alpha(S))$ as before};
		\item \changed{If $\alpha$ is a limit ordinal with  $\cf(\alpha)<\lambda$, then} 
			\begin{equation*}
			\changed{\Tr_\alpha=\bigcap_{\beta\in e_\alpha}\Tr_\beta(S)}
			\end{equation*}
		\item \changed{If $\alpha$ is a limit ordinal with  $\cf(\alpha)=\lambda$, then}
			\begin{equation*}
			\changed{\Tr_\alpha=\bigtriangleup_{\beta\in e_\alpha} \Tr_\beta(S)}
			\end{equation*}
	\end{enumerate}
\end{definition}

\begin{definition}
\changed{Suppose $\lambda$ is an inaccessible cardinal, and let $S$ denote the set of inaccessibles below $\lambda$. For an ordinal $0<\alpha<\lambda^+$, we say that $\lambda$ is {\em$\alpha$-Mahlo} if $\Tr_\beta(S)$ is stationary for every $\beta<\alpha$. We define the {\em Mahlo degree of $\lambda$}, denoted by  $\mathrm{Mahlo}(\lambda)$, to be the least ordinal $\alpha$ such that $\Tr_\alpha(S)$ is non-stationary. For $\lambda$ not inaccessible, we will simply say that $\Mahlo(\lambda)=-1$ }
\end{definition}

\changed{We should note that our definition of the iterated trace operator on a set $S$ will change depending on what ordinal $S$ is a subset of. This could potentially cause confusion, but in practice it will be clear where $S$ lives. At this point, we would like to state several useful properties of the operator $\Tr_\alpha(S)$}.

\begin{lemma}\label{Trace props}
\changed{Suppose $\lambda$ is an ordinal of uncountable cofinality, and $S\subseteq\lambda$}.

\begin{enumerate}
\item \changed{If $\delta\in\Tr_\alpha(S)$ for $0<\alpha<\lambda$, then $\cf(\delta)\geq\aleph_\alpha$.}
\item \changed{If $0<\alpha<\lambda$, then $\Tr_\alpha(\lambda)=S^\lambda_{\geq \aleph_\alpha}$.}
\item \changed{If $\lambda$ is an inaccessible cardinal, then $Tr_\lambda(\lambda)=\{\delta<\lambda : \delta \text{ is inaccessible}\}$.}
\item \changed{Modulo clubs, the definition of $\Tr_\alpha(S)$ does not depend on our choice of $\bar e$.}
\item \changed{The sets $\Tr_\alpha(S)$ form a $\subseteq_{NS}$-decreasing sequence for $0<\alpha$.}
\item \changed{Suppose $\nu\leq\lambda$ are ordinals of uncountable cofinality, and $\gamma<\min\{\cf(\nu),\cf(\lambda)\}$. Then}

\begin{equation*}
\changed{\Tr_{\gamma}(S)\cap \nu=_{\mathrm{NS}_\nu}\Tr_{\gamma}(S\cap\nu).}
\end{equation*}

\item \changed{Suppose $\nu\leq\lambda$ are inaccessible cardinals, and $\gamma=\lambda\times n +\beta$ for $n<\omega$ and $\beta<\nu$. Then}

\begin{equation*}
\changed{\Tr_{\lambda\times n + \beta}(S)\cap \nu=_{\mathrm{NS}_\nu}\Tr_{\nu\times n + \beta}(S\cap\nu).}
\end{equation*}

\item \changed{If $\lambda$ is an inaccessible cardinal, then $\Tr_{\lambda+\beta}(S)=_{NS}\Tr_\beta(\Tr_\lambda(S))$ for $\beta<\lambda\times\omega$.}
\end{enumerate}
\end{lemma}

\begin{proof}
\changed{As the proofs of several of these facts are routine, we will only prove a few while providing brief outlines for the others.}\\

\changed{(1): This is done by induction on $0<\alpha<\lambda$. Of course, if $\delta\in \Tr_1(S)=\Tr(S)$, then $S\cap\delta$ is stationary and so $\cf(\delta)>\omega$. In the case that $\delta\in\Tr_{\alpha+1}(S)=\Tr(\Tr_\alpha(S))$, then $S^{\delta}_{\geq \aleph_\alpha}$ is stationary in $\delta$ by induction and so $\cf(\delta)\geq{\aleph_{\alpha+1}}$. Finally, if $\alpha$ is a limit ordinal, and $\delta\in \Tr_\alpha(S)$ tells us that $\delta\in\Tr_\beta(S)$ for every $\beta<\alpha$ and so $\cf(\delta)\geq \aleph_\beta$ for each $\beta<\alpha$ by induction. Thus, $\cf(\delta)\geq\aleph_\alpha$.}\\

\changed{(2): Note that (1) already tells us that  $\Tr_\alpha(\lambda)\subseteq S^\lambda_{\geq \aleph_\alpha}$. The other inclusion is proved by induction on $\alpha<\lambda$.}\\

\changed{(3): Suppose that $\lambda$ is an inaccessible cardinal. By the previous fact, we have that} 

\begin{equation*}
\changed{\Tr_\lambda(\lambda)=\bigtriangleup_{\alpha<\lambda}\Tr_\alpha(\lambda)=\bigtriangleup_{\alpha<\lambda}S^\lambda_{\geq\aleph_\alpha}.}
\end{equation*}
\\
\changed{But then, $\delta\in \bigtriangleup_{\alpha<\lambda}S^\lambda_{\geq\aleph_\alpha}$ if and only if $\delta\in \bigcap_{\alpha<\delta} S^\lambda_{\geq\aleph_\alpha}$. That is, $\Tr_\lambda(\lambda)$ is precisely the set of $\delta<\lambda$ satisfying $\cf(\delta)=\aleph_\delta$, i.e. the set of inaccessibles below $\lambda$.}\\

\changed{(4): For a regular, uncountable cardinal $\lambda$, we can prove this by induction on $\alpha<\lambda^+$ using the fact that the non-stationary ideal is normal. When $\lambda$ is singular, we simply need to use (2) to verify that the trace operator terminates to a non-stationary set before completeness becomes an issue. }\\

\changed{(5): By induction again, noting that we get genuine inclusion prior to stage $\lambda$, and that we only get $\subseteq_{NS}$ at limit stages above $\lambda$ because $\Tr_\alpha(S)$ depends on the set $e_\alpha$ modulo clubs.}\\

\changed{(6): We omit the proof of this fact and instead note that the proof is similar to the one given below.}\\

\changed{(7): Let $\nu\leq\lambda$ be inaccessible cardinals. We proceed by induction on $\gamma<\lambda\times\omega$. In light of (4), we note that we may assume, for $\gamma=\lambda\times n + \beta$, that}

\begin{equation*}
e_\gamma=\begin{cases}
                                   \{\lambda\times n + \alpha : \alpha<\beta\} & \text{if $\beta>0$} \\
                                   \{\lambda\times (n-1) + \alpha : \alpha<\lambda\} & \text{if $\beta=0$}.
  \end{cases}
\end{equation*} 
\\
\changed{Since we only want equivalence modulo clubs, this will not present an issue. In fact, this is the source of the equivalence modulo clubs instead of genuine equivalence in the statement of this fact. Now, for the base case of our induction, simply note that:}
\begin{equation*}
\changed{\Tr_1(S)\cap\nu=\Tr(S)\cap\nu=\{\alpha<\nu : S\cap\alpha\text{ is stationary}\}=\Tr(S\cap\nu)}.
\end{equation*}
\\
\changed{At successor stages, the computation is similar, so we omit it. At limit stages, suppose that $\gamma=\lambda\times n +\beta$ for $n<\omega$ and a limit ordinal $\beta<\nu$, and further that, for $\lambda\times m + \epsilon<\gamma$ with $\epsilon<\nu$:}

\begin{equation*}
\Tr_{\lambda\times m + \epsilon}(S)\cap\nu=_{\mathrm{NS}_\nu}\Tr_{\nu\times m + \epsilon}(S\cap\nu). 
\end{equation*}
\\
\changed{We have two cases depending on whether or not $\beta=0$, as that dictates the sort of intersection we use. First assume that $\beta\neq 0$, so then}

\begin{align*}
\Tr_{\lambda\times n + \beta}(S)\cap\nu&=\bigcap_{\alpha<\beta}\Tr_{\lambda\times n + \alpha}(S)\cap\nu\\
&=_{\mathrm{NS}_\nu}\bigcap_{\alpha<\beta}\Tr_{\nu\times n + \alpha}(S\cap\nu)\\
&=\Tr_{\lambda\times n + \beta}(S\cap\nu).
\end{align*}
\\
\changed{Here, the first and last equality follow from the definition of $e_{\gamma}$ as well as the definition of the iterated trace operator, while the middle equality follows from the inductive assumption. To be more precise, we take the clubs $C_\alpha$ witnessing that $\Tr_{\lambda\times n + \alpha}(S)\cap\nu=_{\mathrm{NS}_\nu}\Tr_{\nu\times n + \alpha}(S\cap\nu)$, and their interesection $C$ witnesses that middle equality. Finally, assume that $\beta=0$ (so in particular $n>0$). Then}

\begin{align*}
\Tr_{\lambda\times n}(S)\cap\nu&=\bigtriangleup_{\alpha<\lambda}\Tr_{\lambda\times (n-1) + \alpha}(S)\cap\nu\\
&=\{\alpha<\lambda : \alpha\in\bigcap_{\beta<\alpha}\Tr_{\lambda\times(n-1)+\beta}(S)\}\cap\nu\\
&=\{\alpha<\nu : \alpha\in\bigcap_{\beta<\alpha}\Tr_{\lambda\times(n-1)+\beta}(S)\cap\nu\}\\\
&=_{\mathrm{NS}_\nu}\{\alpha<\nu : \alpha\in\bigcap_{\beta<\alpha}\Tr_{\nu\times(n-1)+\beta}(S\cap\nu)\}\\
&=\bigtriangleup_{\alpha<\nu}\Tr_{\lambda\times (n-1) + \alpha}(S\cap\nu)\\
&=\Tr_{\nu\times n}(S\cap\nu).
\end{align*}
\\
\changed{While many of these inequalities follow directly from various definitions, the only suspect one is the fourth one above. For all $\alpha< \nu$, we know that} 

\begin{equation*}
\changed{\bigcap_{\beta<\alpha}\Tr_{\lambda\times(n-1)+\beta}(S)\cap\nu=_{\mathrm{NS}_\nu}\bigcap_{\beta<\alpha}\Tr_{\nu\times(n-1)+\beta}}
\end{equation*}

\changed{from the inductive assumption, and by intersecting witnessing clubs again. For each such $\alpha$, we let $C_\alpha$ witness the above equality, and the diagonal intersection of the sets $C_\alpha$ will witness the fourth equality above.}
\\

\changed{(8): By induction again, using an argument similar to the one given above.}
\end{proof}

\changed{Given (2) and (3) of the lemma above, for $\lambda$ not inaccessible and $S\subseteq \lambda$ we can define $\Tr_\alpha(S)=\emptyset$ for ordinals $\alpha\geq\lambda$.} 

\begin{definition}
For $S\subseteq \lambda$, we will define \changed{\em the $\lambda$ rank of $S$}, $\rk_\lambda(S)$, to be the least $\alpha<\lambda^+$ for which $\Tr_\alpha(S)$ is non-stationary. \changed{If $\lambda$ is of countable cofinality or a successor, then we set $\rk_\lambda(S)=0$.}
\end{definition}

\changed{We will frequently leverage the fact that the rank of a set $S$ reflects downwards in the following sense.}

\begin{lemma}\label{rank below}
Suppose that $\lambda$ is a limit ordinal of uncountable cofinality, then $0<\rk_\lambda(S)\leq \lambda\times n + \beta$ for $n<\omega$ and $\beta<\lambda$ if and only if

\begin{equation*}
\{\delta<\lambda : \rk_\delta(S\cap\delta)\geq\delta\times n +\beta\}
\end{equation*}
\\
is non-stationary in $\lambda$.
\end{lemma}

\begin{proof}
\changed{Suppose that $\lambda$ is a limit ordinal of uncountable cofinality, $S\subseteq\lambda$ is stationary, $n<\omega$, and $\beta<\lambda$. We begin by noting that it follows from (2), (6), and (7) of \cref{Trace props} that, for club almost every $\delta<\lambda$, $\Tr_{\lambda\times n + \beta} (S)\cap \delta$ is non stationary if and only if $\Tr_{\delta\times n + \beta} (S\cap \delta)$ is non-stationary. We now have several cases depending on what $\beta$ is.}
\\

\changed{\underline{Case 1 ($\beta=\gamma+1$):} From the definition, we know that $\rk_\lambda(S)\leq \lambda\times n +\gamma +1$ if and only if} 

\begin{align*}
\Tr(\Tr_{\lambda\times n +\gamma}(S))&=\{\delta<\lambda : \Tr_{\lambda\times n +\gamma}(S)\cap\delta\text{ is stationary}\}\\
&=_{\mathrm{NS}}\{\delta<\lambda :  \Tr_{\delta\times n +\gamma}(S\cap\delta)\text{ is stationary}\}\\
&=\{\delta<\lambda : \rk_\delta(S\cap\delta)>\delta\times n +\gamma\}\\
&=\{\delta<\lambda : \rk_\delta(S\cap\delta)\geq\delta\times n +\gamma+1\}
\end{align*}
\\
\changed{is non-stationary in $\lambda$.  Note that the second equality comes from the observation made at the beginning of the proof.}\\

\changed{\underline{Case 2 ($\beta$ is a non-zero limit ordinal):} Again, it follows from the definition that $\rk_\lambda(S)\leq \lambda\times n +\beta$ if and only if}

\begin{align*}
\Tr_{\lambda\times n + \beta}(S)&=_{\mathrm{NS}}\bigcap_{\alpha<\beta}\Tr_{\lambda\times n + \alpha}(S)\\
&=_{\mathrm{NS}}\bigcap_{\alpha<\beta}\Tr_{\lambda\times n + \alpha+1}(S)\\
&=\{\delta <\lambda : (\forall \alpha<\beta)(\Tr_{\lambda\times n + \alpha}(S)\cap\delta\text{ is stationary})\}\\
&=_{\mathrm{NS}}\{\delta <\lambda : (\forall \alpha<\beta)(\Tr_{\delta\times n + \alpha}(S\cap \delta)\text{ is stationary})\}\\
&=\{\delta <\lambda : (\forall \alpha<\beta)(\rk_\delta(S\cap\delta)>\delta\times n + \alpha\}\\
&=\{\delta <\lambda : \rk_\delta(S\cap\delta)\geq\delta\times n + \beta\}
\end{align*} 
\\
\changed{is non-stationary in $\lambda$.}\\

\changed{\underline{Case 3 ($\beta=0$ and $n>0$):} Again, it follows from the definition that $\rk_\lambda(S)\leq \lambda\times n$ if and only if}

\begin{align*}
\Tr_{\lambda\times n}(S)&=_{\mathrm{NS}}\bigtriangleup_{\alpha<\lambda}\Tr_{\lambda\times (n-1) + \alpha}(S)\\
&=_{\mathrm{NS}}\bigtriangleup_{\alpha<\lambda}\Tr_{\lambda\times(n-1) + \alpha+1}(S)\\
&=\{\delta <\lambda : (\forall \alpha<\delta)(\Tr_{\lambda\times (n-1) + \alpha}(S)\cap\delta\text{ is stationary})\}\\
&=_{\mathrm{NS}}\{\delta <\lambda : (\forall \alpha<\delta)(\Tr_{\delta\times(n-1) + \alpha}(S\cap \delta)\text{ is stationary})\}\\
&=\{\delta <\lambda : (\forall \alpha<\delta)(\rk_\delta(S\cap\delta)>\delta\times (n-1) + \alpha)\}\\
&=\{\delta <\lambda : \rk_\delta(S\cap\delta)\geq\delta\times n\}
\end{align*} 
\\
\changed{is non-stationary in $\lambda$.}\\
\end{proof}

\begin{corollary}
\changed{If $\lambda$ is an inaccessible cardinal, then $\mathrm{Mahlo(\lambda)}\leq\lambda\times n + \beta$ for some $n<\omega$ and $\beta<\lambda$ if and only if }
	
	\begin{equation*}
	\changed{\{\delta<\lambda : \mathrm{Mahlo}(\delta)\geq\delta\times n +\beta\}}
	\end{equation*}
	\\
	\changed{is non-stationary in $\lambda$.}
\end{corollary}

Using ranks, we can introduce an increasing chain of ideals extending the non-stationary ideal which were first considered in \cite{Sh}.

\begin{definition}
	For $\gamma<\lambda^+$, define the ideals
	
	\begin{align*}
	I^{\rk}_\gamma&=\{S\subseteq\lambda : \Tr_\gamma(S)\in\mathrm{NS}_\lambda\},\\
	I^{\rk}_{<\gamma}&=\{S\subseteq\lambda :(\exists\beta<\gamma) (\Tr_\beta(S)\in\mathrm{NS}_\lambda)\}.
	\end{align*}
\end{definition}

An easy observation provides with an alternative characterization of Mahlo-degree. For the sake of notation, we will say that $\lambda$ is $0$-Mahlo if it is inaccessible.

\begin{lemma}\label{Rank Idelas}
	$\lambda$ is $\gamma$-Mahlo for some \changed{$\gamma<\lambda\times \omega$} if and only if $I^{\rk}_{<\lambda+\gamma}$ is proper. 
\end{lemma}

\begin{proof}
	Let $\gamma<\lambda\times\omega$ be given, and let $S$ denote the set of inaccessibles below $\lambda$. Recall that $\lambda$ is $\gamma$-Mahlo precisely when $\Tr_\beta(S)$ is stationary for every $\beta<\gamma$. \changed{By (3) of \cref{Trace props}}, we know that  $\Tr_\lambda(\lambda)=S$, and so 

\begin{equation*}
\changed{\Tr_{\lambda+\beta}(\lambda)=_{NS}\Tr_\beta(\Tr_\lambda(\lambda))=\Tr_\beta(S)}
\end{equation*}
\\
 for every ordinal $\beta<\gamma$ \changed{by (8) of the same lemma.} 
	
	 \changed{Suppose now that $\lambda$ is $\gamma$-Mahlo. Then $\Tr_\beta(S)=_{NS}\Tr_{\lambda+\beta}(\lambda)$ is non-stationary and so it follows that $I^{\rk}_{<\lambda+\gamma}$ is proper.} On the other hand, if $I^{\rk}_{\lambda+\gamma}$ is proper, then $\Tr_{\lambda+\beta}(\lambda)\changed{=_{NS}}\Tr_\gamma(S)$ is non-stationary and so $\lambda$ is $\gamma$-Mahlo.
\end{proof}

At this point, we would like to consider the question of how complete these rank ideals must be. Using ideas of Shelah from \cite{Sh413}, we can show that some of these ideals are closed under increasing unions of length $\theta$ by way of bounding the Galvin-Hajnal rank of functions from $\theta$ to $\lambda$. With this in mind, we take a slight detour to prove some facts about the Galvin-Hajnal rank modulo the bounded ideal.

\begin{definition}
	\changed{Given an uncountable set $A$, a countably complete ideal $I$ on $A$, and an ordinal-valued function $f:A\to\mathrm{ON}$, we define {\em the Galvin-Hajnal rank of $f$ modulo $I$} to be the ordinal}
	
	\begin{equation*}
	\changed{\|f\|_I=\sup\{\|g\|+1 : g<_I f\}.}
	\end{equation*}
\end{definition}

\changed{As we will only be using the ideal of bounded sets when we talk about the Galvin-Hajnal rank, we will omit the subscript denoting the relevant ideal.}

\begin{definition}
	Let $\lambda$ be an ordinal, and let $\omega<\theta<\lambda$ be regular. Define
	
	\begin{equation*}
	\GH(\theta, \lambda)=\min\{\alpha : (\forall f\in{}^\theta\lambda)(||f||<\alpha)\}.
	\end{equation*}	
\end{definition}

\begin{lemma}
	If $\lambda$ is inaccessible, and $\GH(\theta,\lambda)\leq\lambda$, then for every $\delta<\lambda$ we have that $\GH(\theta,\delta)<\lambda$.
\end{lemma}

\begin{proof}
	Suppose that the conclusion fails for some $\delta<\lambda$, and for each $\alpha<\lambda$, fix a function $f_\alpha\in{}^\theta\delta$ such that $||f_\alpha||>\alpha$. Let $c_\delta$ be the function which maps every ordinal $i<\theta$ to $\delta$, and note that $c_\delta\in{}^\theta\lambda$. As $c_\delta$ everywhere-dominates each $f_\alpha$, it follows that $||c_\delta||>\alpha$ for every $\alpha<\lambda$. Thus, $||c_\delta||\geq\lambda$, which is a contradiction.
\end{proof}

\begin{lemma}
	Suppose that $\lambda$ is inaccessible and $\GH(\theta,\lambda)=\lambda$, then the set
	
	\begin{equation*}
	E=\{\delta<\lambda : \GH(\theta,\delta)=\delta\}
	\end{equation*}
	\\
	is $>\theta$ closed and unbounded in $\lambda$.
\end{lemma}

\begin{proof}
	We show that $E$ is $\changed{>}\theta$-closed, and note that a similar idea can be used to show that $E$ is unbounded. Fix $\kappa<\lambda$ regular with $\theta<\kappa$, and let $\langle \delta_i : i<\kappa\rangle$ be a sequence of ordinals from $E$. Set $\delta=\sup_{i<\kappa}\delta_i$, and let $f\in{}^{\theta}\delta$. Note that since $\cf(\delta)>\theta$, the range of $f$ is bounded in $\delta$ and so $f\in{}^\theta\delta_i$ for some $i<\kappa$. But then, $\|f\|<\delta_i$ and so $\delta\in E$.
\end{proof}

It is worth noting that what we're calling $\GH(\theta,\lambda)$ is simply the rank of the constant function $f: \theta\to \{\lambda\}$. With the above in hand, we are now in a position to prove the following result, which can be viewed as what Shelah really proved in Claim 0.20 of \cite{Sh413}.

\begin{theorem}\label{Rank Bound}
	Assume that $\lambda$ is inaccessible, $\theta<\lambda$ is regular such that $\GH(\theta,\lambda)= \lambda$, and $\gamma=\lambda\times n^*+\beta^*<\lambda\times\omega$ (so $\beta^*<\lambda$, and $n<\omega$). If $\langle S_i : i<\theta\rangle$ is a $\subseteq$-increasing sequence of subsets of $S^{\lambda}_{>\theta}$ such that $\rk_\lambda(S_i)<\gamma$ for each $i<\theta$, then
	
	\begin{equation*}
	\rk_\lambda(\bigcup_{i<\theta}S_i)<\lambda\times n^* + \GH(\theta,\beta^*)<\lambda\times (n^*+1).
	\end{equation*}
\end{theorem}

\begin{proof}
	We first note that we may assume there is a least ordinal $i^*$ such that $\rk_\lambda(S_j)>0$ for all $i^*\leq j< \theta$. Otherwise, the union is non-stationary and we have nothing to prove. So we may assume, by taking the union of the first $i^*$ elements and reindexing, that $\langle S_i : i<\theta\rangle$ is a $\subseteq$-increasing sequence of stationary subsets of $S^\lambda_{>\theta}$.
	
	Set $S=\bigcup_{i<\theta}S_i$. For each $i<\theta$, let $\rk_\lambda(S_i)=\lambda\times n_i +\epsilon_i>0$, where $n\leq n^*$ and $\epsilon_i <\lambda$. \changed{By \cref{rank below}, we know that for each $i<\theta$, there is a club} $C_i\subseteq\lambda$ \changed{such} that
	
	\begin{equation*}
	\rk_\delta(S_i\cap\delta):=\delta\times n_{\delta, i} +\epsilon_{\delta, i}<\delta\times n_i +\epsilon_i
	\end{equation*}
	\\
	for each $\delta\in C_i$. \changed{Let $C=\bigcap_{i<\theta}$, and note that, by replacing $S$ with $S\cap\acc(C)$, we may assume that the above inequality holds for all $i<\theta$ and $\delta<\lambda$.} For $\delta\leq\lambda$ and $n\leq n^*$, we define (letting $n_{\lambda, i}=n_i$)
	
	\begin{equation*}
	A^\delta_n=\{i<\theta : n=n_{\delta, i}\}.
	\end{equation*}
	\\
	
	For a fixed $\delta\leq\lambda$, the collection $\{A^\delta_n : n\leq n^*\}$ forms a partition of $\theta$, which means there is some index $n(\delta)$ such that $A^\delta_{n(\delta)}$ is unbounded in $\theta$. On the other hand, the sets $S_i\cap\delta$ are increasing and so, for $i\leq j<\theta$, we immediately know that $n_{i,\delta}\leq n_{j,\delta}$. This tells us that $A^\delta_{n(\delta)}$ is in fact co-bounded and that $A^\delta_n$ is bounded for each $n\neq n(\delta)$. Next, for each $\delta\leq\lambda$, we define the function $f^\delta:A^\delta_{n(\delta)}\to \delta$ by setting $f^\delta(i)=\epsilon_{\delta, i}$. In other words, $f^\delta$ maps each index $i$ in $A^\delta_{n(\delta)}$ to the $<\delta$-portion of the rank of $S_i\cap\delta$. Note also that, as $A^\delta_{n(\delta)}$ is co-bounded, $f^\delta$ is defined on all of $\theta$ modulo the bounded ideal.
	
	At this point, we would like to highlight some properties of the functions $f^\delta$ \changed{for $\delta\in \Tr(S)$}. With that in mind, we fix \changed{such a $\delta$} and for each $i<\theta$, we fix a club $e_i$ of $\delta$ such that the following holds for every $\alpha\in e_i$:
	
	\begin{equation*}
	\rk_{\alpha}(S_i\cap\alpha)<\alpha\times n_{\delta, i} +\epsilon_{\delta, i}.
	\end{equation*}
	\\
	As $\cf(\delta)>\theta$ (\changed{recall that $S\subseteq S^\lambda_{>\theta}$}), it follows that $e=\bigcap_{i<\theta}e_i$ is club in $\delta$. Of course, since $n_{i,\alpha}<n_{i,\delta}$ for each $\alpha\in e$ and $i<\theta$, we immediately see that $n(\alpha)\leq n(\delta)$ \changed{for each $\alpha\in e$}. Next, suppose that $n(\alpha)=n(\delta)$ and that $i\in A^\alpha_{n(\alpha)}\cap A^\delta_{n(\delta)}$. In this case, the inequality
	
	\begin{equation*}
	\alpha\times n(\alpha)+\epsilon_{\alpha, i}<\alpha\times n(\delta)+\epsilon_{\delta, i}
	\end{equation*}
	\\
	tells us that $\epsilon_{\alpha, i}<\epsilon_{\delta, i}$. Hence, $f^\delta(i)<f^\alpha(i)$. So, if $\alpha\in e$ and $n(\alpha)=n(\delta)$, then
	
	\begin{equation*}
	f^\alpha<_{bd}f^\delta.
	\end{equation*}
	\\
	Thus, under these circumstances, we have that $||f^\alpha||<||f^\delta||$. At this point, we have the necessary preliminary observations to prove the inequality in the statement of the theorem. Let 
	
	\begin{equation*}
	E=\{\delta<\lambda : \GH(\theta,\delta)=\delta\}.
	\end{equation*}
	\\
\changed{We claim that}

	\begin{equation*}
	\changed{\rk_\lambda(S)\leq \lambda\times n(\lambda)+||f^\lambda||.}
	\end{equation*}
	\\
To prove this, we will show that the following holds for each $\delta\leq\lambda$:
	
	\begin{equation*}
	\rk_\delta(S\cap E\cap\delta)\leq \delta\times n(\delta)+||f^\delta||.
	\end{equation*}
	\\
Note that this suffices since $S\subseteq S^\lambda_{<\theta}$, and so the fact that $E$ is $>\theta$-club tells us that $S\cap E=_{NS}S$. We proceed by induction on $\delta\leq\lambda$, and note that for $\delta<\lambda$, we may as well assume that $\delta\in \Tr(S)\cap\acc(E)$. Otherwise, $\rk_\delta(S\cap E\cap \delta)=0$ and the inequality holds trivially. 
	
	Suppose now (by way of contradiction) that the inequality above fails for some $\delta$, and let $\delta\leq\lambda$ be the least ordinal for which this happens. Necessarily, $\delta\in \acc(E)\cap\Tr(S)$. \changed{Using an argument from before}, we can find a club $e_0\subseteq\delta$ such that the following holds for each $i<\theta$ and $\alpha\in e_0$:
	
	\begin{equation*}
	\rk_\alpha(S_i\cap\alpha)<\alpha\times n_{\delta, i}+\epsilon_{\delta, i}.
	\end{equation*}
	\\
	Let $e=E\cap e_0$, which is $\changed{>}\theta$-club in $\delta$ since $\delta\in\acc(E)$. As $\rk_\delta(S\cap E\cap\delta)>\delta\times n(\delta)+||f^\delta||$, it follows from \cref{rank below} that the set
	
	\begin{equation*}
	B=\{\alpha\in e: \rk_\alpha(S\cap E\cap \alpha)\changed{\geq}\alpha\times n(\delta)+||f^\delta||\}
	\end{equation*}
	\\
	is stationary in $\delta$. By our inductive assumption, however, we also know that
	
	\begin{equation*}
	\rk_\alpha(S\cap E \cap\alpha)\leq\alpha\times n(\alpha)+||f^\alpha||
	\end{equation*}
	\\
	for each $\alpha\in B$. Thus, we have the following inequality for each $\alpha\in B$:
	
	\begin{equation*}
	\alpha\times n(\delta)+||f^\delta||\changed{\leq}\alpha\times n(\alpha)+||f^\alpha||.
	\end{equation*}
	\\
	
	Now, we note that since each $\alpha\in B$ is in $E$ and $\delta\in E$, we know that $||f^\delta||<\delta$ and $||f^\alpha||<\alpha$. So, the only way that the above inequality holds is if $n(\delta)\leq n(\alpha)$. However, we have already established that $\alpha\in e$ entails that $n(\alpha)\leq n(\delta)$ and thus it follows that $n(\delta)=n(\alpha)$. Therefore, in order for the above inequality to hold, it must be that 
	
	\begin{equation*}
	||f^\delta||\changed{\leq}||f^\alpha||.
	\end{equation*}
	\\
	This is, by our earlier remarks, absurd.
	
	As $n(\lambda)\leq n^*$ and $||f^\lambda||<\lambda$, we get
	
	\begin{equation*}
	\rk_\lambda(S)=\rk_\lambda(S\cap E)\leq \lambda\times n(\lambda)+||f^\lambda||.
	\end{equation*}
	
	We now have two cases to consider: Either $n(\lambda)<n^*$, or $n(\lambda)=n^*$. If $n(\lambda)<n^*$, then 
	
	\begin{align*}
		\rk_\lambda(S)&\leq \lambda\times n(\lambda)+||f^\lambda||\\
		&<\lambda\times n^*\\
		&<\lambda\times n^*+\GH(\theta,\beta^*).
	\end{align*}
	\\
	On the other hand if $n(\lambda)=n^*$, then for all $i\in A^\lambda_{n(\lambda)}$, we know that $\epsilon_i<\beta^*$ and thus $f^\lambda\in{}^\theta\beta^*$. This immediately gives us that
	
	\begin{align*}
	\rk_\lambda(S)&\leq \lambda\times n(\lambda)+||f^\lambda||\\
	&=\lambda\times n^*+||f^\lambda||\\
	&< \lambda\times n^*+\GH(\theta,\beta^*).
	\end{align*}
\end{proof}

Finally, we connect the above theorem to J\'onsson cardinals by way of a result due to Shelah from \cite{Sh}.

\begin{lemma}[Claim 2.12 of Chapter IV of \cite{Sh}]\label{Galvin Hajnal}
	If $\lambda$ is an inaccessible J\'onsson cardinal, then there are unboundedly-many regular $\theta<\lambda$ such that $\GH(\theta,\lambda)=\lambda$.
\end{lemma}

Now suppose that $\lambda$ is inaccessible and J\'onsson, and let $\theta$ be regular such that $\GH(\theta,\lambda)=\lambda$. Fix $\gamma=\lambda\times n^* +\beta^*<\lambda\times\omega$, and let $J$ denote the collection of $S\subseteq S^\lambda_{>\theta}$ that are covered by increasing unions of length $\theta$ of elements from $I^{\rk}_\gamma$. The above theorem tells us then that 

\begin{equation*}
J\subseteq I^{\rk}_{<\lambda\times n^* + \GH(\theta,\beta^*)}\subseteq I^{\rk}_{<\lambda\times (n^*+1)}.
\end{equation*}
\\
Further, one can easily check that $J$ is an ideal. Thus, provided $ I^{\rk}_{<\lambda\times n^* + \GH(\theta,\beta^*)}$ is proper, it follows that $J$ is proper as well. We will make use of this observation in Section 5.

\section{J\'onsson Filters and Stationary Reflection}

In this section, we introduce the notion of a \changed{co}-J\'onsson filter and discuss several of their properties. Our investigations here are inspired by club guessing ideals that appear in the presence of a J\'onsson successor of a singular cardinal. It turns out that the duals of these ideals are \changed{co}-J\'onsson filters, and further that \changed{co}-J\'onsson filters carry many of the combinatorial consequences of the existence of a J\'onsson successor of a singular. 

\begin{definition}
\blfootnote{\changed{Originally, the author referred to these filters as J\'onsson filters. It was pointed out to the author that the notion of a J\'onsson filter appears in the literature with a different definition. As a co-J\'onsson filter is a filter on the codomain of a coloring, the change in name seemed reasonable.} } 
	Suppose that $\lambda$ is a J\'onsson cardinal and that $F$ is a uniform filter on $\lambda$. Then we say that $F$ is a \changed{\em co-J\'onsson filter} (depending on $x$) if there is an $x$ such that, for any $x$-J\'onsson model $M\prec H(\chi)$, we have that $M\cap\lambda\in F$. We say that an ideal $I$ is a \changed{co}-J\'onsson ideal (depending on $x$) if its dual $I^*$ is a \changed{co}-J\'onsson filter.
\end{definition}

The reason for limiting ourselves to \changed{co}-J\'onsson filters that depend on a particular parameter $x$ is because this is how \changed{co}-J\'onsson filters arise in practice. As we mentioned earlier, club guessing ideals provided us with the inspiration for investigating \changed{co}-J\'onsson filters. In fact, Eisworth first showed that these ideals satisfy the above property in \cite{Eis02} relative to $x=\{\bar C, S\}$-J\'onsson models where $S$ is a stationary set and $\bar C$ is a club guessing sequence.

Note that, if there is a \changed{co}-J\'onsson filter on $\lambda$ depending on some parameter $x$, then the intersection of any two $x$-J\'onsson models is again a J\'onsson model. Intuitively, this tells us that J\'onsson models are all very similar. Next we note that if $\lambda$ is a J\'onsson cardinal, and $F$ is a \changed{co}-J\'onsson filter as witnessed by $x$, we can define the minimal \changed{co}-J\'onsson filter coming from $\chi$ and $x$.

\begin{definition}
	Suppose that $\lambda$ is J\'onsson. Then we define the filter $F_x$ as the filter generated by sets of the form $M\cap\lambda$ where $M\prec H(\chi)$ is an $x$-J\'onsson model.
\end{definition}

It is worth noting that we have defined \changed{co}-J\'onsson filters using elementary submodels, but we could just as easily have utilized colorings.

\begin{lemma}\label{Coloring Equivalence}
	If $\lambda$ is J\'onsson, and $F$ is a uniform filter on $\lambda$, then $F$ is a \changed{co}-J\'onsson filter if and only if there is a coloring $G:[\lambda]^{<\omega}\to\lambda$ such that, for every $A\in[\lambda]^{\lambda}$, we have that $\ran(G\upharpoonright[A]^{<\omega})\in F$.
\end{lemma}

\begin{proof}
	First suppose that $F$ is a \changed{co}-J\'onsson filter, as witnessed by $\chi$ and $x\in H(\chi)$. Let $\mathfrak{A}=(H(\chi),\in, x, <_\chi)$ and enumerate the Skolem functions of $\mathfrak{A}$ by $\langle f_n : n<\omega\rangle$. We may arrange it so that each $f_n$ is $k_n$-ary where $k_n\leq n$. Using this, we can define the usual Skolem coloring $G$ by setting
	
	\begin{equation*}
	G(\alpha_1,\ldots,\alpha_n)=
	\begin{cases}
	f_n(\alpha_1,\ldots,\alpha_{k_n}) & \text{if this is an ordinal below }\lambda\\
	0 & \text{otherwise}
	\end{cases}
	\end{equation*}
	\\
	For $A\in[\lambda]^\lambda$, set $M_A=Sk^\mathfrak{A}(A)$, and note that $\ran(G\upharpoonright[A]^{<\omega})=\Tr(M_A)$ for each such $A$. As $\lambda, x\in M_A$, and $|M_A\cap \lambda|=\lambda$, it follows that either $M_A$ is a J\'onsson model or $\lambda\subseteq M_A$. In either case, we get that $M_A\cap\lambda\in F$.
	\\
	
	For the other direction, suppose that we have \changed{a} coloring $G:[\lambda]^{<\omega}\to\lambda$ such that, for every $A\in[\lambda]^{\lambda}$, the set $\ran(G\upharpoonright[A]^{<\omega})\in F$. Let $\chi>\lambda$ be sufficiently large and regular, set $x=\{G\}$, and fix a $(\chi, x)$-J\'onsson model $M$. Since $G\in M$, we see that $\ran(G\upharpoonright[M\cap\lambda]^{<\omega})\subseteq M\cap\lambda$, which immediately gives us that $M\cap\lambda\in F$. As $M$ was an arbitrary $(\chi, x)$-J\'onsson model, we see that $\chi$ and $x$ witness that $F$ is a \changed{co}-J\'onsson filter.
\end{proof}

Suppose that we have a \changed{co}-J\'onsson filter on a J\'onsson cardinal $\lambda$. If we \changed{let} $G$ be the coloring witnessing that this filter is J'onsson \changed{then for any $A\in[\lambda]^\lambda$, we get that $\ran(G\upharpoonright[A]^{<\omega})=_F\lambda$}. \changed{In that sense, we can think of $G$ as a witness to the fact that $\lambda$ is not J\'onsson modulo $F$ and so witnesses a near failure of J'onsonness.} On the other hand, \changed{co-}J\'onsson ideals are responsible for a number of large cardinal-type properties that J\'onsson successors of singulars enjoy. Part of the blame for this lies with a property known as $\theta$-irregularity, which can be traced back to problems in model theory.

\begin{definition}
	Suppose $I$ is an ideal on a set $A$, and $\theta$ is a regular cardinal. We say that $I$ is \changed{\em $\theta$-irregular} if the following holds: Whenever $\langle A_i : i<\theta\rangle$ is a sequence of $I$-positive sets, there is some $H\in[\theta]^\theta$ such that 
	
	\begin{equation*}
	\bigcap_{i\in H}A_i\neq\emptyset.
	\end{equation*}
\end{definition}

It is worth pointing out that the author mistakenly referred to the property above as $\theta$-regularity in an earlier work, \cite{Ahm}, as we will make use of several results from that paper. At this point, we would like to show that a \changed{co-}J\'onsson ideal on a J\'onsson cardinal $\lambda$ is always $\theta$-irregular for many regular $\theta$ below $\lambda$. We do this by way of partitioning irregularity into two separate properties.

\begin{definition}
	An ideal $I$ over some set $A$ is \changed{\em weakly $\theta$-saturated} for some cardinal $\theta$ if there is no partition of $A$ into $\theta$-many disjoint, $I$-positive pieces.
\end{definition}

One may think of weak saturation as a measure of how close the ideal $I$ is to being maximal. Note that if $\theta_0<\theta_1$ and $I$ is weakly $\theta_0$-saturated, then of course it will be weakly $\theta_1$-saturated. 

\begin{definition}
	An ideal $I$ over some set $A$ is \changed{\em $\theta$-indecomposable} for some regular cardinal $\theta$ if $I$ is closed under increasing unions of length $\theta$.
\end{definition}

We would like to remark that, unlike weak saturation, indecomposability is neither upwards nor downward hereditary. In fact, $\kappa$-completeness of an ideal $I$ is equivalent to the statement that $I$ is $\theta$-indecomposable for every regular $\theta<\kappa$.

\begin{lemma}[Proposition 2.6 of \cite{Eis3}]\label{lem: reg}
	Let $I$ be an ideal on a set $A$. The following are equivalent for a regular cardinal $\theta$.
	
	\begin{enumerate}
		\item $I$ is weakly $\theta$-saturated and $\theta$-indecomposable.
		\item Whenever $\langle B_i : i<\theta\rangle$ is a $\subseteq$-increasing $\theta$-sequence of subsets of $A$, then there is some $j^*<\theta$ such that
		
		\begin{equation*}
		j^*\leq j<\theta\quad\implies\quad B_j=_I \bigcup_{i<\theta}B_i.
		\end{equation*}
		\item $I$ is $\theta$-irregular.
	\end{enumerate}
\end{lemma}

So in order \changed{to} show that any \changed{co-}J\'onsson ideal is $\theta$-irregular, it suffices to show that it is $\theta$-indecomposable and weakly $\theta$-saturated. 

\begin{lemma}
	Suppose that $\lambda$ is a J\'onsson cardinal, and that $I$ is a \changed{co-}J\'onsson ideal over $\lambda$. If $I$ is not weakly $\theta$-saturated, then there exists a coloring $c:[\lambda]^{<\omega}\to\theta$ with no weakly homogeneous set.
\end{lemma}

The proof is identical to the proof of Corollary 5.28 of \cite{Eis1}, but we include it for the sake of completeness.

\begin{proof}
	Let $\langle A_i : i<\theta\rangle$ be a sequence of disjoint $I$-positive sets witnessing that $I$ is not weakly $\theta$ saturated. Let $I^*$ denote the filter dual to $I$, and let $G:[\lambda]^{<\omega}\changed{\to\lambda}$ witness that $I^*$ is a \changed{co-}J\'onsson filter. We define a coloring $c:[\lambda]^{<\omega}\to \theta$ by setting $c(\alpha_1,\ldots, \alpha_n)$ to be the unique $i<\theta$ such that $G(\alpha_1,\ldots,\alpha_n)\in A_i$.
	
	If $A\in[\lambda]^\lambda$, then $\ran(G\upharpoonright[A]^{<\omega})\in F$ and hence meets every element of $A_i$ for each $i<\theta$. But then, $\ran(c\upharpoonright[A]^{<\omega})=\theta$ by construction.
\end{proof}

From this, we get the following corollary.

\begin{corollary}\label{Weak Saturation}
	If $I$ is a \changed{co-}J\'onsson ideal over a J\'onsson cardinal $\lambda$, then $I$ is weakly $\theta$-saturated for some $\theta<\lambda$. If $\lambda=\mu^+$, then we can find such a $\theta$ below $\mu$.
\end{corollary}

\begin{proof}
	We first assume that $\lambda$ is limit and suppose otherwise, i.e. that there is no $\theta<\lambda$ such that $I$ is weakly $\theta$-saturated. By the previous lemma, this means that $\lambda\not\to[\lambda]^{<\omega}_\theta$ for unboundedly-many $\theta<\lambda$. Let $\chi$ be large enough and regular so that $H(\chi)$ knows this, and let $M\prec H(\chi)$ be a J\'onsson model. For each $\theta\in M\cap\lambda$, fix a coloring $c_\theta:[\lambda]^{<\omega}\to\theta$ in $M$ which witnesses $\lambda\not\to[\lambda]^{<\omega}_\theta$. But then, $\ran(c_\theta\upharpoonright [M\cap\lambda]^{<\omega})\subseteq M$ for each such $\theta$, and so $\lambda\subseteq M$ which is a contradiction.
	
	Note that if $\lambda=\mu^+$ for $\mu$ singular, we only need to show that $\mu\subseteq M$. With this in mind, note that we may have chosen $M$ be a J\'onsson model with $\cf(\mu)\subseteq M$. Since $\cf(\mu), \mu\in M$, we can find a map $f:\cf(\mu)\to\mu$ in $M$ with image cofinal in $\mu$. Thus, $\mu\cap M$ is unbounded in $\mu$, and we instead get that $\mu\subseteq M$ by the same argument as above. 
\end{proof}

In \cite{Sh}, it is shown that the following lemma holds for certain types of \changed{co-}J\'onsson ideals (\changed{Claim 2.10 of chapter IV}). It turns out that this holds for any  \changed{co-}J\'onsson ideal.

\begin{lemma}
	Assume that $\lambda$ is J\'onsson and that $I$ is a \changed{co-}J\'onsson ideal on $\lambda$, as witnessed by a parameter $x$. For $\sigma<\lambda$ regular, if $I$ is not $\sigma$-indecomposable and $M\prec H(\chi)$ is an $x\cup\{I\}$-J\'onsson model with $\sigma\in M$, then $|M\cap\sigma|=\sigma$.
\end{lemma}

\begin{proof}
	As $I$ is not $\sigma$-indecomposable, we can find a sequence $\langle A_i : \changed{i}<\sigma\rangle$ such that $\bigcup_{i<\sigma}A_i\notin I$, but for every $w\in[\sigma]^{<\sigma}$, we have $\bigcup_{i\in \sigma}A_i\in I$. Set $A=\bigcup_{i<\sigma}A_i$. By disjointifying $A$, we can build a function $h:A\to \sigma$ such that, for every $w\in[\sigma]^{<\sigma}$,
	
	\begin{equation*}
	\{\alpha\in A : h(\alpha)\notin w\}\notin I
	\end{equation*} 
	\\
	Now let $M$ be a $(\chi,x\cup\{I\})$-J\'onsson model with $\sigma\in M$. By elementarity, we may assume that $A,h\in M$. 
	
	We claim that $|M\cap \sigma|=\sigma$. Otherwise, $M\cap\sigma=w\in[\sigma]^{<\sigma}$, and so 
	
	\begin{equation*}
	B=h^{-1}[\sigma\setminus M]=\{\alpha\in A : h(\alpha)\notin w\}\notin I
	\end{equation*} 
	\\
	As $B\notin I$, we can find some $\gamma\in M\cap\lambda\cap B$. However, we have that $h\in M$, and so $h(\gamma)\in M$ which is absurd.
\end{proof}

\begin{lemma}\label{Indecomposable}
	Suppose that $\lambda$ is a J\'onsson cardinal, and that $I$ is a \changed{co-}J\'onsson ideal (possibly depending on $x$) over $\lambda$. Then $I$ is $\theta$-indecomposable for unboundedly-many $\theta\in\changed{\lambda\cap\REG}$.
\end{lemma}

\begin{proof}
	We first note if $\kappa$ is a regular cardinal, then $\kappa^+$ is not J\'onsson \changed{since} \cref{Tryba-Woodin} tells us that every stationary subset of a regular J\'onsson cardinal reflects. Now let $M\prec H(\chi)$ be an $x\cup \{I\}$-J\'onsson model over $\lambda$, and suppose that $I$ is $\theta$-indecomposable for only boundedly-many $\theta\in\changed{\lambda\cap\REG}$. By elementarity, $|\alpha|^{++}\in M$ for every $\alpha\in M$. As $\lambda$ is unbounded in $M$, it thus follows that
	
	\begin{equation*}
	X=\{\sigma\in M : \sigma=\kappa^+\text{ for }\kappa\text{ regular, and }I\text{ is not }\sigma\text{-indecomposable}\}
	\end{equation*} 
	\\
	is unbounded in $\lambda\cap\REG$. By the previous lemma, we know that $|M\cap\sigma|=\sigma$ for each $\sigma\in X$. Each such $\sigma$ is not J\'onsson, and so $\sigma\subseteq M$. \changed{In the case that $\lambda$ is limit}, it follows immediately from the fact that $X$ is unbounded in $\lambda$ that $\lambda\subseteq M$ which contradicts the fact that $M$ was a J\'onsson model. \changed{On the other hand, if $\lambda=\mu^+$ for $\mu$ singular, then $X$ is instead unbounded in $\mu$, which tells us that $\mu\subseteq M$. But then, we can use this to recover all of $\mu^+$ as a subset of $M$, again contradicting the fact that $M$ is a J\'onsson model.}
\end{proof}

From this, it immediately follows that if $\lambda$ is a J\'onsson cardinal, and $I$ is a \changed{co-}J\'onsson ideal on $\lambda$, then $I$ is both $\theta$-indecomposable and weakly $\theta$-saturated for unboundedly-many cardinals in $\lambda\cap \mathrm{REG}$. 

\begin{corollary}
	Suppose that $\lambda$ is a J\'onsson cardinal. If $I$ is a \changed{co-}J\'onsson ideal on $\lambda$, then $I$ is $\theta$-irregular for unboundedly-many regular $\theta$ in $\lambda\cap\REG$.
\end{corollary}

As we mentioned earlier, $\theta$-irregularity lends us some large cardinal-type properties. In particular, we can obtain simultaneous reflection. Given an ideal $I$ on $\lambda$, we fix some notation.

\begin{enumerate}
	\item $\mathrm{Comp}(I)$ is the largest cardinal $\theta$ such that $I$ is $\theta$-complete.
	\item $\mathrm{Reg}(I)=\{\theta < \lambda : I\text{ is }\theta\text{-irregular}\}$.
	\item $S(I)=\{\alpha<\lambda  : \cf(\alpha)<\alpha \text{ and }\cf(\alpha)\in\mathrm{Reg}(I)\}$.\\
\end{enumerate}

\begin{theorem}[Eisworth, \cite{Eis3}]\label{Simultaneous Reflection}
	Let $I$ be an ideal on a cardinal $\lambda$. If there is a cardinal $\theta$ such that $I$ is $\theta$-irregular, then
	
	\begin{enumerate}
		\item $S(I)$ is stationary, and
		\item Any fewer than $\mathrm{Comp}(I)$-many stationary subsets of $S(I)$ reflect simultaneously.
	\end{enumerate}
\end{theorem}

This immediately gives us the following corollary.

\begin{corollary}
	Suppose $\lambda$ is J\'onsson. If there is a J\'onsson ideal $I$ on $\lambda$, then any fewer than $\mathrm{Comp}(I)$-many stationary subsets of $S(I)$ reflect simultaneously. 
\end{corollary}

Next, we establish a connection between J\'onsson filters and club guessing. We include the proof of the below result, as it serves as a prototypical example of how one produces J\'onsson Filters. We would like to point out that the theorem is essentially shining a light on what Lemma 1.8 and Lemma 1.9 from Chapter III of  \cite{Sh} are saying.

\begin{theorem}[Essentially Shelah]\label{idpci is Jonsson}
	Suppose that $\lambda$ is a regular J\'onsson cardinal. If $S\subseteq\lambda$ is a stationary set not reflecting at inaccessibles, and $\bar C=\langle C_\delta : \delta\in S\rangle$ is a club guessing sequence, then \changed{$I(\bar C, \NS_\lambda)$} is a $\{\bar C, S\}$-J\'onsson ideal on $\lambda$.
\end{theorem}

\begin{proof}
	Let $x=\{\bar C, S\}$, and let \changed{$M$ be an $x$-J\'onsson model}. \changed{Note that since $I(\bar C, \NS_\lambda)$ is a proper ideal extending the $\NS_\lambda$, we only need to show that $\lambda\setminus (M\cap\lambda)\in I(\bar C, \NS_\lambda)$.} Let
	
	\begin{equation*}
	E=\{\delta<\lambda :\delta=\sup(M_\cap\delta)\},
	\end{equation*}
	\\
	and note that $E$ is club in $\lambda$. As $\bar C$ is a club guessing sequence the set
	
	\begin{equation*}
	T=\{\delta\in S : C_\delta\cap E\notin I_\delta\}
	\end{equation*}
	\\
	is stationary in $\lambda$.
	
	At this point, we would like to invoke \cref{Swallowing Ladders}, but the problem is that we first need to show that $\cf(\beta_{M}(\delta))<\delta$ for each $\delta\in S\setminus M$. \changed{With that in mind, we fix $\delta\in S\setminus M$ and begin by noting that $\beta_M(\delta)$ is a limit ordinal of uncountable cofinality. We first assume that $\lambda=\mu^+$ for $\mu$ singular, so we may assume that $S\subseteq\mu^+\setminus \mu$ without disturbing the club guessing. In this case, we know that $\cf(\beta_M(\delta))<\mu$ but also that $\mu<\delta$, and so we are done.} 

\changed{Now assume that $\lambda$ is inaccessible, and note that we may assume that each $\delta\in S$ is a cardinal by removing a non-stationary set from $S$. We first claim that $\beta_M(\delta)$ is a limit cardinal. To see this, we note that if $\beta_M(\delta)$ fails to be a cardinal, then $|\beta_M(\delta)|<\beta_M(\delta)$. On the other hand, the fact that $\beta_M(\delta)\in M$ tells us that $|\beta_M(\delta)|\in M$ and so it follows from definition of $\beta_M(\delta)$ that $|\beta_M(\delta)|<\delta=|\delta|<\beta_M(\delta)$, which is absurd. Similarly, if $\beta_M(\delta)$ is a succesor cardinal, then its predecessor is definable in $M$ and would have to be below $\delta$, which is itself a cardinal. So, $\beta_M(\delta)$ is a limit cardinal.} 
	
\changed{At this point, we only need to show that $\cf(\beta_M(\delta))<\beta_M(\delta)$ since the fact that $\cf(\beta_M(\delta))$ is definable in $M$ will give us that $\cf(\beta_M(\delta))<\delta$. With this in mind, we will show that $S$ reflects at $\beta_M(\delta)$, which will tell us that $\beta_M(\delta)$ cannot be inaccessible and hence must be singular since it is a limit cardinal. If $S$ fails to reflect at $\beta_M(\delta)$, since both $S$ and $\beta_M(\delta)$ are in $M$ it follows from elementarity that there is some} $d\subseteq\beta_M(\delta)$ be club with $d\in M$ and \changed{$S\cap\beta_M(\delta)=\emptyset$. But it follows from} \cref{Reflection at beta delta} that $\delta$ is an accumulation point of $d$ and hence in $d$,\changed{ which is a contradiction since then $\delta\in S\cap\delta$. Thus, for $\delta\in S\setminus M$, we have that $\cf(\beta_M(\delta))<\delta$.}
	
	Fix $\delta \in S\cap E$. By \cref{Swallowing Ladders}, we know that either\\
	
	\begin{enumerate}
		\item $\delta\in M$ and thus $\nacc(C_\delta)\cap E\subseteq M$, or
		\item $\delta\notin M$ and thus $\nacc(C_\delta)\cap E\cap \mathrm{cof}(>\cf(\beta_M(\delta)))\subseteq M$.\\
	\end{enumerate}
	
	\changed{So if $\delta\in M$, we immediately see that $\nacc(C_\delta)\cap E\cap(\lambda\setminus (M\cap\lambda))$ is empty. On the other hand, if $\delta\notin M$, then every $\alpha\in \nacc(C_\delta)\cap E\cap(\lambda\setminus (M\cap\lambda))$ satisfies $\cf(\alpha)<\cf(\beta_M(\delta))<\delta$ and thus $\nacc(C_\delta)\cap E\cap(\lambda\setminus (M\cap\lambda))\in I_\delta$. From this, it follows that $\lambda\setminus (M\cap\lambda)\in I(\bar C, \NS_\lambda)$.}
\end{proof}

Note that If $\lambda$ is inaccessible then we can find a stationary set not reflecting at inaccessible \changed{assuming that} $\lambda$ is not $\omega$-Mahlo. We close this section by noting that the $\theta$-irregularity of J\'onsson ideals afford us some tools from pcf theory. In order to do this, we quote a consequence of Theorem 2.1 of \cite{Ahm}.

\begin{theorem}\label{trichotomy}
	Suppose that $A$ is a set of ordinals, and $I$ is an ideal on $A$. Suppose that $\theta<\lambda$ are regular cardinals such that $I$ is $\theta$-indecomposable and weakly $\theta$-saturate. If $\vec{f}=\langle f_\xi : \xi <\lambda\rangle$ is a $<$-increasing sequence of functions from $A$ to $\mathrm{ON}$, then $\vec{f}$ has an exact upper bound (\changed{modulo $I$}) $f\in {}^A\mathrm{ON}$ with $\{a\in A : \cf(f(a))<\kappa\}\in I$ for every regular $\kappa<\lambda$.
\end{theorem}

\begin{theorem}\label{Jonssons below}
	If $\lambda$ is an inaccessible J\'onsson cardinal such that there is a \changed{co}-J\'onsson filter on $\lambda$, then the regular J\'onsson cardinals are unbounded in $\lambda\cap\mathrm{REG}$.
\end{theorem}

\begin{proof}
	We proceed by contradiction, so assume that the set of regular J'onsson cardinals below $\lambda$ is bounded in $\lambda$. Suppose that $I$ is a \changed{co-J\'onsson ideal on $\lambda$ relative to} some parameter $x$, and note that $I$ is compatible with the ideal of bounded subsets of $\lambda$ since $M\cap\lambda$ must be unbounded for any J\'onsson model $M$. Let $J$ be the ideal generated by the ideal of bounded sets from $I$. Note that $J$ is still a co-J'onsson ideal, and thus $\theta$-irregular for some $\theta<\lambda$. 
	
	For each $\alpha<\lambda$ we let $f_\alpha$ be the constant function at $\alpha$. Then, $\bar f=\langle f_\alpha : \alpha<\lambda\rangle$ is a $<$-increasing sequence of functions from $\lambda$ to the ordinals. Thus, \cref{trichotomy} gives us a $J$-eub $f$ of $\bar f$ such that
	
	\begin{equation*}
	\{ \changed{\alpha<\lambda}: \cf(f(\alpha))<\changed{\kappa}\}\in J
	\end{equation*}
	\\
	for every regular $\changed{\kappa}<\lambda$. Note that $f(\alpha)$ is regular for $J$ almost every $\alpha<\lambda$, otherwise the function $\alpha\mapsto \cf(f(\alpha))$ is below $f$ on a $J$-positive set. But then, there is some $\alpha^*<\lambda$ such that $\cf(f(\alpha))<\alpha^*$ on a $J$-positive set which is a contradiction. Next, note that the cofinality condition on $f$ forces that $\ran(f\upharpoonright A)$ us unbounded in $\lambda$ for each $A\in J^+$.
	
	Now let $M$ be an $x$-J\'onsson model with $f\in M$, and define a function $g:\lambda\to\lambda$ by $g(\alpha)=\sup(M\cap f(\alpha))$. First, we claim that $g<_J f$. Otherwise, there is a $J$\changed{-positive} set $A$ such that the following holds for each $\alpha\in A$:
	
	\begin{enumerate}
		\item $\sup(M\cap f(\alpha))=f(\alpha)$;
		\item $f(\alpha)$ is regular;
		\item $\alpha\in M$. 
	\end{enumerate}
	
	Further, $\ran(f\upharpoonright A)$ is unbounded and a subset of $M$, since $f\in M$. So we have an unbounded set of regular cardinals fo the form $f(\alpha)$ in $M$ which are not J\'onsson, but satisfy $|f(\alpha)\cap M|=f(\alpha)$. From this, it follows that each such $f(\alpha)\subseteq M$ and thus that $\lambda\subseteq M$ which is a contradiction. Thus, $g<_J f$, and so there is some $\alpha^*<\lambda$ such that $\sup(M\cap f(\alpha))<\alpha^*$ on a $J$-large set. This is a contradiction.
\end{proof}

\changed{We note that the above theorem did not require that the co-J\'onsson ideal comes from club guessing. Much like the case for a J\'onsson successor of a singular, we can prove a similar theorem from club guessing sequences directly.}

\begin{theorem}\label{small Jonsson}
\changed{Suppose that $\lambda$ is an inaccessible J\'onsson cardinal. If there is a co-J\'onsson ideal $I$ (perhaps depending on a parameter $x$) on $\lambda$ extending the non-stationary ideal, then the set of regular non-J\'onsson cardinals below $\lambda$ is in $I$.}
\end{theorem}

\begin{proof}
\changed{Suppose otherwise, i.e. that $A=\{\alpha<\lambda : \alpha\text{ is regular and not J\'onsson}\}$ is not in $I$. Let $M$ be a J'onsson model with $x\in M$, and let}

\begin{equation*}
\changed{E=\{\delta<\lambda : \delta=\sup(M\cap\delta)\}.}
\end{equation*}
\\
\changed{Since $I$ extends the non-stationary ideal, we know that $E\cap M\cap A$ is $I$-positive, and hence unbounded in $\lambda$. But for every $\alpha\in E\cap M\cap A$, we know that $\alpha$ is a regular non-J\'onsson cardinal in $M$ with $|M\cap\alpha|=\alpha$, and thus $\alpha\subseteq M$. Since $E\cap M\cap A$ is unbounded, it follows then that $\lambda\subseteq M$ which is a contradiction.}
\end{proof}

\section{The Mahlo Degree of an Inaccessible J\'onsson Cardinal}

At this point, we would like to make more clear the relationship between Mahlo rank and J\'onssonness. We first note that the techniques employed in \cref{Jonssons below} can be used to put further restrictions on which sorts of cardinals can carry J\'onsson ideals. We note that the proof given below borrows from Claim 3.3 of Chapter III from \cite{Sh}.

\begin{lemma}[Shelah]\label{Mahlo Jonsson}
	If $\lambda$ is inaccessible and J\'onsson, then $\lambda$ must be at least Mahlo.
\end{lemma}

\begin{proof}
	Assume otherwise, and let $\lambda$ be an inaccessible J\'onsson cardinal which is not Mahlo. Then, by the comments proceeding \cref{idpci is Jonsson}, we can find a J\'onsson ideal $I$ on $\lambda$ which extends the non-stationary ideal. As in the proof of \cref{Jonssons below}, we let $f_\alpha$ be the constant function at $\alpha$ for each $\alpha<\lambda$, and let $\bar f=\langle f_\alpha : \alpha<\lambda\rangle$. We can then find a \changed{I}-eub $f$ for $\bar f$ such that 
	
	\begin{equation*}
	\{\changed{\alpha<\lambda}: \cf(f(\alpha))<\changed{\kappa}\}\in \changed{I}
	\end{equation*}
	\\
	for every $\changed{\kappa}\lambda$ regular. It suffices to show that $f(\alpha)$ is inaccessible for $I$-almost every $\alpha<\lambda$ and that $\ran(f)$ is stationary. We already know that $f(\alpha)$ is a regular cardinal for $I$-almost every $\alpha<\lambda$ by the argument in \cref{Jonssons below}. 
	
	Suppose, for the sake of contradiction, that the set 
	
	\begin{equation*}
	X=\{\alpha<\lambda : f(\alpha)\text{ is a successor cardinal}\}.
	\end{equation*}
	\\
	is $I$-positive. Define a function $g:X\to\mathrm{ON}$ by setting $g(\alpha)$ equal to the cardinal predecessor of $f(\alpha)$. Since $g<_{I\upharpoonright X} f\upharpoonright X$, it follows that there is some $\alpha^*<\lambda$ such that $g(\alpha)<\alpha^*$ for almost every $\alpha\in X$. This is, of course, a contradiction since that would mean $f(\alpha)$ is bounded in an $I$-positive set.
	
	To see that $\ran(f)$ is stationary, suppose otherwise and let $E\subseteq\lambda$ be club with $E\cap \ran(f)=\emptyset$. Define a function $g:E\to \lambda$ by setting $g(\alpha)=\max(f(\alpha)\cap E)$, and noting of course that $g<_I f$ since $I$ extends the non-stationary ideal. Now, for each $i<\lambda$, let $i_E=\min(E\setminus (i+1))$, and note that if $f(\alpha)>i_E$, then $g(\alpha)>i$. This, however, tells us that $g$ must be unbounded, which contradicts the fact that $f$ is an eub for $\bar f$.
\end{proof}

\changed{Note that when we applied \cref{Swallowing Ladders} in the proof of \cref{idpci is Jonsson}, we first showed that $\beta_M(\delta)$ was often singular for $\delta\in S\setminus M$. This allowed us to show that $M\cap\lambda$ swallowed significant portions of the ladder system $\bar C$. Our goal now is to show something similar happens in the case that $\lambda$ is inaccessible and J\'onsson, provided $\mathrm{Mahlo}(\lambda)<\lambda\times\omega$. For the remainder of this section, we will assume that $\Tr_\gamma(S)$ has been defined using the clubs}

\begin{equation*}
e_\gamma=\begin{cases}
\{\lambda\times n + \alpha : \alpha<\beta\} & \text{if $\beta>0$} \\
\{\lambda\times (n-1) + \alpha : \alpha<\lambda\} & \text{if $\beta=0$}
\end{cases}
\end{equation*} 
\\
\changed{for $\gamma<\lambda\times\omega$ limit.}

\begin{lemma}
	Let $\lambda$ be an inaccessible J\'onsson cardinal, and let $M$ be a J\'onsson model for $\lambda$. Suppose \changed{$\gamma=\lambda\times n + \beta$ for some $n<\omega$ and $\beta<\lambda$, and set}
	
	\begin{equation*}
	A_\gamma=\{\alpha <\lambda :\changed{\mathrm{Mahlo}(\beta_M(\alpha))\geq \beta_M(\alpha)\times n + \beta}\}.
	\end{equation*}
	\\
	Then, $\Tr(A_\gamma)\subseteq_{NS}A_{\gamma+1}$.
\end{lemma}

\begin{proof}
	\changed{Suppose $\lambda$ is an inaccessible J\'onsson cardinal, $\gamma=\lambda\times n +\beta$ is as above, and $M$ is a J\'onsson model for $\lambda$. Note that we may assume that $\Tr_\gamma(A)$ is stationary, otherwise we have noting to prove. With that in mind, let}
	
	\begin{equation*}
	\changed{E=\{\delta<\lambda : \delta=(\sup M\cap \delta)\},}
	\end{equation*}
	
	and fix $\delta\in E$ such that $\delta\cap A_\gamma$ is stationary in $\delta$. We need to show that $\beta_M(\delta)$ is at least $\beta_M(\alpha)\times n + \beta+1$-Mahlo. \changed{Note that if $\beta_M(\delta)=\delta$, then we are done by the definition of $A_\gamma$ and Corollary 3.1}. Suppose otherwise, and let $C\in M$ be club in $\beta_M(\delta)$ such that every $\Mahlo(\alpha)<\alpha\times n +\beta$. \changed{Note that such a club exists since $\beta_M(\delta)\in M$, which tells us that $\Mahlo(\beta_M(\delta))\in M$ as well}. \changed{By \cref{Reflection at beta delta}, we know that} $\delta\in \acc(C)$, i.e. that $\delta\cap C$ is club in $\delta$. \changed{As $\delta\cap A_\gamma$ is stationary, we can find} $\alpha\in \delta\cap A_\gamma\cap C$. \changed{It then follows from \cref{club indiscenibility} that} $\beta_M(\alpha)\in C$, which is a contradiction.
\end{proof}

\begin{lemma}
	Suppose that $\lambda$ is an  inaccessible J\'onsson cardinal such that $\Mahlo(\lambda)<\lambda\times\omega$. Let $M$ be a J\'onsson model for $\lambda$ and, for $\gamma<\lambda\times\omega$, set
	
	\begin{equation*}
	A_\gamma=\{\alpha <\lambda :\changed{\mathrm{Mahlo}(\beta_M(\alpha))\geq \beta_M(\alpha)\times n + \beta}\}.
	\end{equation*}
	\\
	Then, $\Tr_\gamma(A_0)\subseteq_{NS}A_{\gamma}$.
\end{lemma}

\begin{proof}
	\changed{Suppose $\lambda$ is an inaccessible J\'onsson cardinal, and $M$ is a J\'onsson model for $\lambda$.} Note that the previous lemma already gives us that $\Tr(A_0)\subseteq_{NS}A_1$, so suppose that we know $\Tr_\gamma(A_0)\subseteq_{NS}A_\gamma$ for some $\gamma<\lambda^+$. Applying the previous lemma again gives us that
	
	\begin{equation*}
	\Tr_{\gamma+1}\changed{(A_0)}=\Tr(\Tr_\gamma(A_0))\subseteq_{NS}\Tr(A_\gamma)\subseteq_{NS}A_{\gamma+1}.
	\end{equation*}
	\\
	
	At limit stages \changed{$\gamma=\lambda\times n +\beta$, we have two cases: either $\beta$ is a non-zero limit ordinal, or $\beta=0$ and $n>0$.} Assume \changed{first that $\beta$ is a non-zero limit ordinal, and that $\Tr_{\lambda\times n +\alpha}(A)\subseteq_{NS}A_{\lambda\times n +\alpha}$ for every $\alpha<\beta$. For each such $\alpha$, let $E_{\alpha}$ be club in $\lambda$ that witnesses this, and set $E=\bigcap_{\alpha<\beta}E_\alpha$. Suppose now that} 
		
	\begin{equation*}
	\changed{\delta\in \Tr_\gamma(A_0)\cap E=\bigcap_{\alpha<\beta}\Tr_{\lambda\times n+\alpha}(A_0)\cap\bigcap_{\alpha<\beta}E_\alpha.}
	\end{equation*}
	\\
	\changed{In that case, $\delta\in A_{\lambda\times n +\alpha}$ for all $\alpha<\beta$, i.e. $\Mahlo(\beta_M(\delta))\geq \beta_M(\delta)\times n + \alpha$ for each such $\alpha$. But then, $\Mahlo(\beta_M(\delta))\geq \beta_M(\delta)\times n +\beta$.}
	
	\changed{Assume now that $\beta=0$ and $n>0$, and that $\Tr_{\lambda\times (n-1) +\alpha}(A)\subseteq_{NS}A_{\lambda\times (n-1) +\alpha}$ for every $\alpha<\lambda$. For each such $\alpha$, let $E_{\alpha}$ be club in $\lambda$ that witnesses this, and set $E=\bigtriangleup_{\alpha<\lambda}E_\alpha$. Suppose now that} 
	
	\begin{equation*}
	\changed{\delta\in \Tr_\gamma(A_0)\cap E=\bigtriangleup_{\alpha<\lambda}\Tr_{\lambda\times (n-1)+\alpha}(A_0)\cap\bigtriangleup_{\alpha<\lambda}E_\alpha.}
	\end{equation*}
	\\
	\changed{In that case, $\delta\in A_{\lambda\times (n-1) +\alpha}$ for all $\alpha<\delta$, i.e. $\Mahlo(\beta_M(\delta))\geq \beta_M(\delta)\times (n-1) + \alpha$ for each such $\alpha$. But then, $\Mahlo(\beta_M(\delta))\geq \beta_M(\delta)\times (n-1) +\delta$. If $\Mahlo(\beta_M(\delta))>\beta_M(\delta)\times n$, then we are done. Otherwise, we know that }
	
	\begin{equation*}
	\changed{\beta_M(\delta)\times (n-1) +\delta\geq\Mahlo(\beta_M(\delta))=\beta_M(\delta)\times (n-1) +\beta*\leq\beta_M(\delta)\times n.}
	\end{equation*}
	\\
	\changed{For $\beta^*\leq \beta_M(\delta)$. Since $\beta^*$ is definable in $M$, it follows that either $\beta^*=\delta$ and thus that $\beta_M(\delta)=\delta$, or $\beta^*=\beta_M(\delta)$. In either case, we get that  $\Mahlo(\beta_M(\delta))=\beta_M(\delta)\times n$, i.e. that $\delta\in A_\gamma$.}
\end{proof}

\begin{corollary}\label{AM}
	Suppose that $\lambda$ is an  inaccessible J\'onsson cardinal such that $\Mahlo(\lambda)<\lambda\times\omega$. Let $M$ be a J\'onsson model for $\lambda$, and set
	
	\begin{equation*}
	A=\{\alpha < \lambda : \beta_M(\alpha)\text{ is inaccessible}\}.
	\end{equation*}
	\\
	The rank of $A$ is at most the Mahlo-degree of $\lambda$.
\end{corollary}

\begin{proof}
	\changed{Assume} otherwise, and let $C\in M$ be club in $\lambda$, witnessing that the Mahlo degree of $\lambda$ is below $\gamma=\mathrm{rk}(A)$. \changed{Let}
	
	\begin{equation*}
	\changed{A_\gamma=\{\alpha <\lambda :\changed{\mathrm{Mahlo}(\beta_M(\alpha))\geq \beta_M(\alpha)\times n + \beta}\}.}
	\end{equation*}
	\\
	\changed{From the previous lemma}, we know that $\Tr_\gamma(A)\subseteq_{NS}A_\gamma$, and so it follows that $A_\gamma$ is stationary. Let $\alpha\in A_\gamma\cap C$, and note then that $\beta_M(\alpha)\in C$ by Lemma 2.3, which is a contradiction \changed{by the definition of $C$}.\\
	
\end{proof}

At this point, we are ready to give a new proof of the fact that any inaccessible J\'onsson cardinal $\lambda$ must be at least $\lambda\times\omega$-Mahlo. We will need to following club guessing result from \cite{Sh413}.

\begin{theorem}[Claim 0.14 of \cite{Sh413}]\label{Club Guessing}
	Assume that 
	\begin{enumerate}
		\item $\lambda$ is inaccessible, and $\theta<\lambda$ is regular;
		\item $A\subseteq\lambda$ is a stationary set of limit ordinals not reflecting outside of itself;
		\item $J$ is a $\theta$-indecomposable ideal on $\lambda$ extending the non-stationary ideal;
		\item $S\in J^+$, $S\cap A=\emptyset$, and $S\cap\mathrm{cof}(\theta)=\emptyset$.\\
	\end{enumerate}
	
	Then we can find an $S$-club system $\bar C=\langle C_\delta : \delta\in S\rangle$ such that, for every club $E\subseteq \lambda$:
	
	\begin{equation*}
	\{\delta\in S : \delta=\sup(E\cap\nacc(C_\delta)\cap A)\}\in J^+
	\end{equation*}
\end{theorem}

\changed{As in the proof that an inaccessible J\'onsson cardinal must be at least Mahlo, we will show that cardinals of small Mahlo degree must carry a co-J\'onsson ideal and generate a contradiction from this.}

\begin{theorem}
	\changed{Suppose that $\lambda$ is Mahlo and J\'onsson, and set} 
	
	\begin{equation*}
	\changed{A=\{\alpha <\lambda : 0\leq\Mahlo(\alpha)<\alpha\times n + \beta \}.}
	\end{equation*}
	\\
	\changed{If $\Mahlo(\lambda)=\lambda\times n + \beta<\lambda\times \omega$, then there is a co-J\'onsson filter on $\lambda$ which concentrates on $A$ and extends the club filter.}
\end{theorem}

\begin{proof}
	 Assume that $\lambda$ is J\'onsson, and let $\gamma=\lambda\times n^*+\beta^*<\lambda\times\omega$ be the Mahlo degree of $\lambda$. Note that $A$ only reflects at inaccessibles and that, by \cref{rank below}, club almost every inaccessible below $\lambda$ is in $A$. So, by removing a non-stationary subset of $A$, we may assume that $A$ does not reflect outside of itself.  By \cref{Galvin Hajnal}, let  $\theta<\lambda$ be an uncountable regular cardinal such that $\GH(\theta,\lambda)=\lambda$. Next, we note that $A\in I^{\rk}_\gamma$, while $S^\lambda_{>\theta}\notin I^{\rk}_{\lambda+\gamma}$ by \cref{Rank Idelas}, and that $\lambda+\gamma=\lambda\times(n^*+1)+\beta^*>\gamma$. Let $J$ be the $\theta$-indecomposable completion of $I^{\rk}_\gamma$, and recall from the discussion after the proof of \cref{Rank Bound} that $J\subseteq I^{\rk}_{\lambda+\gamma}$. 
	
	\changed{Next, let }
	
	\begin{equation*}
	\changed{A_M=\{\alpha < \lambda : \beta_M(\alpha)\text{ is inaccessible}\},}
	\end{equation*}
	\\
	and note that $A_M\in J$ by \cref{AM}. Thus, it follows that the set
	
	\begin{equation*}
	S=\{\delta<\lambda : \cf(\delta)>\theta,\text{ and }|\delta|=\delta\}\setminus (A\cup A_M)
	\end{equation*}
	\\
	is $J$-positive since $A\in J$ and $S^\lambda_{>\theta}\in J^+$. Thus, $A$, $J$, $S$, $\theta$, and $\lambda$ satisfy the assumptions of \cref{Club Guessing} and so there is an $S$-club system $\bar C=\langle C_\delta : \delta\in S\rangle$ such that 
	
	\begin{equation*}
	\{\delta\in S : \delta=\sup(E\cap\nacc(C_\delta)\cap A)\}\in J^+
	\end{equation*}
	\\
	for every club $E\subseteq \lambda$. \changed{That is, such that $I(\bar C, J)\upharpoonright A$ is proper (since $A$ is a set of inaccessibles, the fact that $E\cap\nacc(C_\delta)\cap A$ is unbounded is enough to guarantee that the set is unbounded in cofinality as well)}. Now, let $M$ be a J\'onsson model for $\lambda$ with $S, \bar C\in M$. Set 
	
	\begin{equation*}
	E=\{\delta<\lambda : \sup(M\cap\delta)=\delta\}.
	\end{equation*}
	\\
	By \cref{Swallowing Ladders}, for every $\delta\in S\cap E$, we have either\\
	
	\begin{enumerate}
		\item $\delta\in M\implies\nacc(C_\delta)\cap E\subseteq M$, or
		\item $\delta\notin M\implies\nacc(C_\delta)\cap E\cap \mathrm{cof}(>\cf(\beta_M(\delta)))\subseteq M$.\\
	\end{enumerate}

	\changed{Suppose that $\delta\in S\cap E$, and note that for all such $\delta$, $\beta_M(\delta)$ is not inaccessible. Thus, we can argue as in the proof of \cref{idpci is Jonsson} to conclude that $\lambda\setminus(M\cap\lambda)\in I(\bar C, J)\upharpoonright A$.}
\end{proof}

\begin{theorem}[Shelah]
	If $\lambda$ is an inaccessible J\'onsson cardinal, then $\lambda$ must be at least $\lambda\times\omega$-Mahlo.
\end{theorem}

\begin{proof}
	\changed{We proceed by induction on the Mahlo degree of $\lambda$. Note that the base case, where $\Mahlo(\lambda)=0$ is taken care of by \cref{Mahlo Jonsson}. So  let $\gamma=\lambda\times n + \beta<\lambda\times\omega$, and suppose that the result holds for all $\gamma'<\gamma$ $\lambda$. Suppose further, for the sake of contradiction, that there is an inaccessible J\'onsson cardinal $\lambda$ with $\Mahlo(\lambda)=\gamma<\lambda\times\omega$. Let }
	
	\begin{equation*}
	\changed{A=\{\alpha <\lambda : 0\leq\Mahlo(\alpha)<\alpha\times n + \beta \}.}
	\end{equation*}
	\\
	\changed{By Theorem 5.2, there is a co-J\'onsson ideal $I$ on $\lambda$ extending the non-stationary ideal with $A\in I^*$. But then, our inductive assumptions tell us that every $\alpha$ in $A$ is a regular non-J\'onsson, which contradicts \cref{small Jonsson}.} 
\end{proof}

\section{Questions}

One thing to note about the previous proof is that we heavily leveraged the fact that $\lambda$ was not $\lambda\times\omega$-Mahlo in order to arrive at a contradiction. In particular, we used the fact that there \changed{is} sizable a gap between the rank of the set of inaccessibles and the rank of $\lambda$. This does not happen when the Mahlo degree of $\lambda$ is at least $\lambda\times\omega$-Mahlo. So there seems to be little hope of pushing this result further by a similar argument using the same notion of rank. This, of course, leads us to our first question.

\begin{Question}
	Is it consistent that there is an inaccessible J\'onsson cardinal with Mahlo degree precisely $\lambda\times\omega$?
\end{Question}

Another interesting thing to note regarding inaccessible J\'onsson cardinals is that the only time we constructed J\'onsosn filters on inaccessible J\'onsson cardinals was for the purpose of generating a contradiction. In fact, we only know how to construct J\'onsson filters on J\'onsson successors of singular cardinals. So, we have no consistent examples of the existence of a J\'onsson filter.

\begin{Question}
	Is it consistent that there is a J\'onsson cardinal $\lambda$ such that $\lambda$ has a J\'onsson filter?
\end{Question}

\bibliographystyle{plain}

\begin{thebibliography}{1}
	
	\bibitem{Ahm}
	Shehzad Ahmed.
	\newblock An Extension of Shelah's Trichotomy Theorem.
	\newblock {\em Archive for Mathematical Logic}, 58:137--153, 2019.
	
	\bibitem{Eis02}
	Todd Eisworth.
	\newblock A Note on J\'onsson Cardinals.
	\newblock {\em Topology Proceedings}, 27:173--178, 2003.
	
	\bibitem{Eis3}
	Todd Eisworth.
	\newblock Club Guessing, Stationary Reflection, and Coloring Theorems.
	\newblock {\em Annals of Pure and Applied Logic}, 161:1216--1243, 2010.
	
	\bibitem{Eis1}
	Todd Eisworth.
	\newblock Successors of Singular Cardinals.
	\newblock In Matthew Foreman and Akihiro Kanamori, editors, {\em The Handbook
		of Set Theory}, volume~2, pages 1229--1350. Springer, 2010.
	
	\bibitem{Jech84}
	Thomas Jech.
	\newblock Stationary Subsets of Inaccessible Cardinals.
	\newblock In Donald A.~Martin James E.~Baumgartner and Saharon Shelah, editors,
	{\em Axiomatic Set Theory}, page 115–142. American Mathematical Society,
	1984.
	
	\bibitem{Ka}
	Akihiro Kanamori.
	\newblock {\em The Higher Infinite: Large Cardinals in Set Theory from Their
		Beginnings}.
	\newblock Springer Monographs in Mathematics. Springer-Verlag, Berlin, second
	edition, 2003.
	
	\bibitem{Sh}
	Saharon Shelah.
	\newblock {\em Cardinal Arithmetic}, volume~29 of {\em Oxford Logic Guides}.
	\newblock Oxford University Press, 1994.
	
	\bibitem{Sh506}
	Saharon Shelah.
	\newblock The pcf Theorem Revisited.
	\newblock {\em Algorithms and Combinatorics}, 14:420--259, 1997.
	
	\bibitem{Sh413}
	Saharon Shelah.
	\newblock More Jonsson Algebras.
	\newblock {\em Archive for Mathematical Logic}, 42:1--44, 2003.
	
\end{thebibliography}

\end{document}